\documentclass[11pt,a4paper]{article} \usepackage{amsmath}
\usepackage{amsthm} \usepackage{amscd} \usepackage{amssymb}
\usepackage{graphicx} 
\usepackage[labelformat=empty]{subfig}
\usepackage{multirow}
\usepackage{geometry}
\usepackage{pstricks}
\newtheorem{theorem}{Theorem}[section]
\newtheorem{lemma}[theorem]{Lemma}
\newtheorem{corollary}[theorem]{Corollary}
\newtheorem{proposition}[theorem]{Proposition}

\newtheorem{conjecture}[theorem]{Conjecture}

\theoremstyle{definition} 
\newtheorem{example}[theorem]{Example}

\theoremstyle{remark} \newtheorem*{remark}{Remark}
 
\newtheorem*{ack}{Acknowledgments}

\newcounter{temp}

\makeatletter \def\square{\RIfM@\bgroup\else$\bgroup\aftergroup$\fi
\vcenter{\hrule\hbox{\vrule\@height.6em\kern.6em\vrule}\hrule}\egroup}
\makeatother
 
\DeclareMathOperator{\Spec}{Spec} \DeclareMathOperator{\Proj}{Proj}

 \DeclareMathOperator{\cone}{cone}
\DeclareMathOperator{\st}{st} \DeclareMathOperator{\link}{lk}
\DeclareMathOperator{\flip}{Fl} 
 
 \DeclareMathOperator{\Def}{Def}

\DeclareMathOperator{\vertices}{vert}

\DeclareMathOperator{\init}{in}
\DeclareMathOperator{\Circ}{circ}
\DeclareMathOperator{\sta}{star}
\DeclareMathOperator{\SL}{SL}
\DeclareMathOperator{\Sp}{Sp}
\DeclareMathOperator{\Spin}{Spin}
\DeclareMathOperator{\Aut}{Aut}

\newcommand{\baseRing}[1]{\ensuremath{\mathbb{#1}}}
 \newcommand{\ZZ}{\baseRing{Z}}
\newcommand{\CC}{\baseRing{C}} \newcommand{\QQ}{\baseRing{Q}}
\newcommand{\A}{\mathcal{A}} \newcommand{\K}{\mathcal{K}}
\newcommand{\PP}{\baseRing{P}}
\newcommand{\GG}{\mathbb{G}}
\newcommand{\mcL}{\mathcal{L}}

\newcommand{\CO}{\mathcal{O}}
\newcommand{\I}{\mathcal{I}}

\newcommand{\trielevenflips}{
\psset{unit=.9cm}
\begin{pspicture}(-5.5,-5.5)(5.5,5)
\psline(1,3)(3,5)
\psline(-1,-3)(-3,-5)
\psline(3,-1)(5,-1)
\psline(3,-1)(1,3)
\psline(-3,1)(-5,1)
\psline(-3,1)(-1,-3)
\psline(1,3)(1,1)(3,-1)
\psline(-1,-3)(-1,-1)(-3,1)
\psline(1,1)(1,-1)(3,-1)
\psline(-1,-1)(1,-1)
\psline(-1,-1)(-1,1)(-3,1)
\psline(1,1)(-1,1)
\psline(1,-3)(1,-1)
\psline(1,-3)(3,-1)
\psline(1,-3)(-1,-3)
\psline(1,-3)(3,-5)
\psline(-1,3)(-1,1)
\psline(-1,3)(-3,1)
\psline(-1,3)(1,3)
\psline(-1,3)(-3,5)
\psline[linestyle=dotted](-1,-3)(1,-1)(-1,1)(1,3)
\psline[linestyle=dotted](1,-3)(-1,-1)(1,1)(-1,3)
\psdots(1,3)(-1,-3)(3,-1)(-3,1)(1,1)(-1,-1)(1,-1)(-1,1)(1,-3)(-1,3)
\psdots[dotstyle=o](3,5)(-3,-5)(5,-1)(-5,1)(3,-5)(-3,5)
\rput(-2.6,-5){$w$}
\rput(2.6,-5){$w$}
\rput(-1.5,-3){$y_{10}$}
\rput(1.6,-3){$y_{11}$}
\rput(-1.4,-1){$x_{10}$}
\rput(1.4,-1.4){$x_{11}$}
\rput(3.2,-1.4){$z_1$}
\rput(5,-1.4){$w$}

\rput(2.6,5){$w$}
\rput(-2.6,5){$w$}
\rput(1.5,3){$y_{01}$}
\rput(-1.6,3){$y_{00}$}
\rput(1.4,1){$x_{01}$}
\rput(-1.4,1.4){$x_{00}$}
\rput(-3.2,1.4){$z_0$}
\rput(-5,1.4){$w$}
\end{pspicture}
}
\newcommand{\tritenflips}{
\psset{unit=.7cm}
\begin{pspicture}(-5.5,-5.5)(5.5,5)
\psdots(0,-1)(0,1)(0,3)(0,-3)(3,3)(-3,-3)(3,0)(-3,0)(5,5)(-5,-5)
\psline(0,-3)(0,3)
\psline(-6,-6)(-3,-3)(0,-1)(3,0)
\psline(6,6)(3,3)(0,1)(-3,0)
\psline(-3,-3)(-3,0)(0,3)(3,3)(3,0)(0,-3)(-3,-3)
\psline(-3,0)(-5,-5)(0,-3)
\psline(3,0)(5,5)(0,3)
\psline[linestyle=dotted](-3,-3)(0,1)(3,0)
\psline[linestyle=dotted](3,3)(0,-1)(-3,0)
\psline[linestyle=dotted,linearc=2](-5,-5)(-5,3)(0,3)
\psline[linestyle=dotted,linearc=2](-5,-5)(3,-5)(3,0)
\psline[linestyle=dotted,linearc=2](5,5)(5,-3)(0,-3)
\psline[linestyle=dotted,linearc=2](5,5)(-3,5)(-3,0)
\rput(-5.7,-4.8){$z_{23}$}
\rput(0,-3.5){$z_{33}$}
\rput(-3.7,0){$z_{22}$}
\rput(2.5,3.5){$z_{11}$}
\rput(-.5,1.5){$z_{12}$}
\rput(5.7,4.8){$y_{23}$}
\rput(0,3.5){$y_{33}$}
\rput(3.7,0){$y_{22}$}
\rput(-2.5,-3.5){$y_{11}$}
\rput(.5,-1.5){$y_{12}$}
\end{pspicture}
}

\newcommand{\trieleven}{
\psset{unit=.4cm}
\begin{pspicture}(-5.5,-5.5)(5.5,5)
\psline(1,3)(3,5)
\psline(-1,-3)(-3,-5)
\psline(3,-1)(5,-1)
\psline(3,-1)(1,3)
\psline(-3,1)(-5,1)
\psline(-3,1)(-1,-3)
\psline(1,3)(1,1)(3,-1)
\psline(-1,-3)(-1,-1)(-3,1)
\psline(1,1)(-1,-1)
\psline(1,1)(1,-1)(3,-1)
\psline(-1,-1)(1,-1)(-1,-3)
\psline(-1,-1)(-1,1)(-3,1)
\psline(1,1)(-1,1)(1,3)
\psline(1,-3)(1,-1)
\psline(1,-3)(3,-1)
\psline(1,-3)(-1,-3)
\psline(1,-3)(3,-5)
\psline(-1,3)(-1,1)
\psline(-1,3)(-3,1)
\psline(-1,3)(1,3)
\psline(-1,3)(-3,5)
\psdots(1,3)(-1,-3)(3,-1)(-3,1)(1,1)(-1,-1)(1,-1)(-1,1)(1,-3)(-1,3)
\psdots[dotstyle=o](3,5)(-3,-5)(5,-1)(-5,1)(3,-5)(-3,5)
\end{pspicture}
}

\newcommand{\triten}{
\psset{unit=.4cm}
\begin{pspicture}(-5.5,-5.5)(5.5,5)
\psline(1,3)(3,5)
\psline(-1,-3)(-3,-5)
\psline(3,-1)(5,-1)
\psline(3,-1)(1,3)
\psline(-3,1)(-5,1)
\psline(-3,1)(-1,-3)
\psline(1,3)(1,1)(3,-1)
\psline(-1,-3)(-1,-1)(-3,1)
\psline(1,1)(-1,-1)
\psline(1,1)(1,-1)(3,-1)
\psline(-1,-1)(1,-1)(-1,-3)
\psline(-1,-1)(-1,1)(-3,1)
\psline(1,1)(-1,1)(1,3)
\psline(1,-3)(1,-1)
\psline(1,-3)(3,-1)
\psline(1,-3)(-1,-3)
\psline(1,-3)(3,-5)

\psline[linestyle=dashed](-3,1)(1,3)
\psdots(1,3)(-1,-3)(3,-1)(-3,1)(1,1)(-1,-1)(1,-1)(-1,1)(1,-3)
\psdots[dotstyle=o](3,5)(-3,-5)(5,-1)(-5,1)(3,-5)
\end{pspicture}
}
\newcommand{\trinine}{
\psset{unit=.4cm}
\begin{pspicture}(-5.5,-5.5)(5.5,5)
\psline(1,3)(3,5)
\psline(-1,-3)(-3,-5)
\psline(3,-1)(5,-1)
\psline(3,-1)(1,3)
\psline(-3,1)(-5,1)
\psline(-3,1)(-1,-3)
\psline(1,3)(1,1)(3,-1)
\psline(-1,-3)(-1,-1)(-3,1)
\psline(1,1)(-1,-1)
\psline(1,1)(1,-1)(3,-1)
\psline(-1,-1)(1,-1)(-1,-3)
\psline(-1,-1)(-1,1)(-3,1)
\psline(1,1)(-1,1)(1,3)

\psline[linestyle=dashed](3,-1)(-1,-3)
\psline(-3,1)(1,3)
\psdots(1,3)(-1,-3)(3,-1)(-3,1)(1,1)(-1,-1)(1,-1)(-1,1)
\psdots[dotstyle=o](3,5)(-3,-5)(5,-1)(-5,1)
\end{pspicture}
}
\newcommand{\trieight}{
\psset{unit=.4cm}
\begin{pspicture}(-5.5,-5.5)(5.5,5)
\psline(1,3)(3,5)
\psline(-1,-3)(-3,-5)
\psline(3,-1)(5,-1)
\psline(3,-1)(1,3)
\psline(-3,1)(-5,1)
\psline(-3,1)(-1,-3)
\psline(1,3)(1,1)(3,-1)
\psline(-1,-3)(-1,-1)(-3,1)
\psline(1,1)(-1,-1)
\psline(1,1)(1,-1)(3,-1)
\psline(-1,-1)(1,-1)(-1,-3)

\psline[linestyle=dashed](-3,1)(1,1)
\psline(3,-1)(-1,-3)
\psline(-3,1)(1,3)
\psdots(1,3)(-1,-3)(3,-1)(-3,1)(1,1)(-1,-1)(1,-1)
\psdots[dotstyle=o](3,5)(-3,-5)(5,-1)(-5,1)
\end{pspicture}
}

\newcommand{\triseven}{
\psset{unit=.4cm}
\begin{pspicture}(-5.5,-5.5)(5.5,5)
\psline(1,3)(3,5)
\psline(-1,-3)(-3,-5)
\psline(3,-1)(5,-1)
\psline(3,-1)(1,3)
\psline(-3,1)(-5,1)
\psline(-3,1)(-1,-3)
\psline(1,3)(1,1)(3,-1)
\psline(-1,-3)(-1,-1)(-3,1)
\psline(1,1)(-1,-1)

\psline[linestyle=dashed](3,-1)(-1,-1)
\psline(-3,1)(1,1)
\psline(3,-1)(-1,-3)
\psline(-3,1)(1,3)
\psdots(1,3)(-1,-3)(3,-1)(-3,1)(1,1)(-1,-1)
\psdots[dotstyle=o](3,5)(-3,-5)(5,-1)(-5,1)
\end{pspicture}
}
\newcommand{\trisix}{
\psset{unit=.4cm}
\begin{pspicture}(-5.5,-5.5)(5.5,5)
\psline(1,3)(3,5)
\psline(-1,-3)(-3,-5)
\psline(3,-1)(5,-1)
\psline(3,-1)(1,3)
\psline(-3,1)(-5,1)
\psline(-3,1)(-1,-3)
\psline(1,3)(1,1)(3,-1)

\psline[linestyle=dashed](-1,-3)(1,1)

\psline(-3,1)(1,1)
\psline(3,-1)(-1,-3)
\psline(-3,1)(1,3)
\psdots(1,3)(-1,-3)(3,-1)(-3,1)(1,1)
\psdots[dotstyle=o](3,5)(-3,-5)(5,-1)(-5,1)
\end{pspicture}
}
\newcommand{\trifive}{
\psset{unit=.4cm}
\begin{pspicture}(-5.5,-5.5)(5.5,5)
\psline(1,3)(3,5)
\psline(-1,-3)(-3,-5)
\psline(3,-1)(5,-1)
\psline(3,-1)(1,3)
\psline(-3,1)(-5,1)
\psline(-3,1)(-1,-3)

\psline[linestyle=dashed](-1,-3)(1,3)

\psline(3,-1)(-1,-3)
\psline(-3,1)(1,3)
\psdots(1,3)(-1,-3)(3,-1)(-3,1)
\psdots[dotstyle=o](3,5)(-3,-5)(5,-1)(-5,1)
\end{pspicture}
}
\newcommand{\trifour}{
\psset{unit=.4cm}
\begin{pspicture}(-5.5,-5.5)(5.5,5)
\psline(-3,1)(-5,1)
\psline(1,3)(3,5)
\psline(3,-1)(5,-1)
\psline[linestyle=dashed](3,-1)(-3,1)
\psline(1,3)(-3,1)
\psline(3,-1)(1,3)
\psdots(-3,1)(1,3)(3,-1)
\psdots[dotstyle=o](-5,1)(3,5)(5,-1)
\end{pspicture}
}

\newcommand{\octa}{
\psset{unit=2.2cm}
\begin{pspicture}(-1.25,-1.25)(1.25,1.25)
\pspolygon[linecolor=gray,fillstyle=solid,fillcolor=lightgray](0,-1)(-.707,-.707)(0,1)(1,0)
\psline(1,0)(.707,.707)(0,1)(-.707,.707)(-1,0)(-.707,-.707)(0,-1)(.707,-.707)(1,0)
\psline(1,0)(-.707,-.707)
\psline(0,-1)(0,1)
\rput(0,1.1){$i$}
\rput(1.1,0){$j$}
\rput(0,-1.1){$k$}
\rput(-.8,-.8){$l$}
\rput(.15,.4){$\delta_{ik}$}
\rput(.5,-.05){$\delta_{jl}$}
\rput(-.2,-.2){$Q$}
\end{pspicture}
}

\newcommand{\octb}{
\psset{unit=2.2cm}
\begin{pspicture}(-1.25,-1.25)(1.25,1.25)
	\pspolygon[linecolor=gray,fillstyle=solid,fillcolor=lightgray](0,-1)(.707,.707)(1,0)(.707,-.707)
	\psline(1,0)(.707,.707)(0,1)(-.707,.707)(-1,0)(-.707,-.707)(0,-1)(.707,-.707)(1,0)
	\psline(.707,-.707)(.707,.707)
	\psline(0,-1)(1,0)
\end{pspicture}
}

\newcommand{\octc}{
\psset{unit=2.2cm}
\begin{pspicture}(-1.25,-1.25)(1.25,1.25)
	\pspolygon[linecolor=gray,fillstyle=solid,fillcolor=lightgray](0,-1)(-.707,-.707)(-1,0)(.707,.707)
	\psline(1,0)(.707,.707)(0,1)(-.707,.707)(-1,0)(-.707,-.707)(0,-1)(.707,-.707)(1,0)
	\psline(-.707,-.707)(.707,.707)
	\psline(0,-1)(-1,0)
\end{pspicture}
}

\newcommand{\octd}{
\psset{unit=2.2cm}
\begin{pspicture}(-1.25,-1.25)(1.25,1.25)
	\pspolygon[linecolor=gray,fillstyle=solid,fillcolor=lightgray](.707,.707)(1,0)(0,-1)(-.707,-.707)
	\psline(1,0)(.707,.707)(0,1)(-.707,.707)(-1,0)(-.707,-.707)(0,-1)(.707,-.707)(1,0)
	\psline(0,-1)(.707,.707)
	\psline(-.707,-.707)(1,0)
\end{pspicture}
}

\def\address#1#2{\begingroup
\noindent\parbox[t]{7.8cm}{%
\small{\scshape\ignorespaces#1}\par\vskip1ex
\noindent\small{\itshape E-mail address}%
\/: #2\par\vskip4ex}\hfill%
\endgroup}%

\date{}
\begin{document}

\title{Degenerations to Unobstructed Fano Stanley-Reisner Schemes}\author{Jan
Arthur Christophersen \and Nathan Owen Ilten}   

\maketitle
\begin{abstract} We construct degenerations of Mukai varieties and
linear sections thereof to special unobstructed Fano Stanley-Reisner
schemes corresponding to convex deltahedra. This can be used to find
toric degenerations of rank one index one Fano threefolds. In the
second part we find many higher dimensional unobstructed Fano and
Calabi-Yau Stanley-Reisner schemes. The main result is that the
Stanley-Reisner ring of the boundary complex of the dual polytope of
the associahedron has trivial $T^2$.
\end{abstract}

\section*{Introduction} In \cite{mu:cuk}, Mukai showed that rank one
index one Fano threefolds of genus $g\leq 10$ appear as complete
intersections in (weighted) projective spaces and homogeneous spaces.
\noindent
	\begin{center}\begin{tabular}{|l| l| l| l|} \hline Name &
Degree & Genus & Embedding\\ \hline $V_2'$& $2$ & $2$&Sextic in
$\PP(1,1,1,1,3)$\\ $V_4'$& $4$ & $3$&Quartic in $\PP^4$\\ $V_6$& $6$ &
$4$&Intersection of quadric and cubic in $\PP^5$\\ $V_8$& $8$ &
$5$&Intersection of three quadrics in $\PP^6$\\ $V_{10}$& $10$ & $6$&
Codim. 2 linear subspace of $M_6:=Q_2\cap G(2,5)$\\ $V_{12}$& $12$ &
$7$& Codim. 7 linear subspace of $M_7:=SO(5,10)$\\ $V_{14}$& $14$ &
$8$& Codim. 5 linear subspace of $M_8:=G(2,6)$\\ $V_{16}$& $16$ & $9$&
Codim. 3 linear subspace of $M_9:=LG(3,6)$\\ $V_{18}$& $18$ & $10$&
Codim. 2 linear subspace of $M_{10}:=\GG_2$\\ \hline
\end{tabular} \end{center} Here $Q_2$ is a generic quadric. The
varieties $M_g$ are called \emph{Mukai varieties}. The homogeneous
spaces involved are the Grassmannians $G(2,5)$ and $G(2,6)$ associated
to $\SL_5\CC$ and $\SL_6\CC$, the (even) orthogonal Grassmannian or
spinor variety $SO(5,10)$ associated to $\Spin_{10}\CC$, the
Lagrangian Grassmannian $LG(3,6)$ associated to $\Sp_6\CC$, and $\GG_2$
which is associated to the adjoint representation of the exceptional
semi-simple Lie Group $G_2$. Note that the $V_{2g-2}$ denote
deformation classes as in the original 
classification of Iskovskih in \cite{is:fa}, see Corollary
\ref{homspace} below.  

In the first part of this paper we compare this series with a special series of Fano
Stanley-Reisner schemes. If $T$ is a combinatorial sphere then the
Stanley-Reisner scheme of the join of $T$ and a simplex is Fano
(Proposition \ref{fanocomp}). In Section \ref{sec:tri} we describe a
series of triangulated $2$-spheres $T_{n}$, $4 \le n \le 11$, with $n$
vertices such that the Stanley-Reisner scheme of the cone over $T_{n}$
is a natural flat degeneration of $V_{2n-4}$. In fact for $6 \le g \le
10$ the Stanley-Reisner scheme of the join of $T_{g+1}$ and a suitable
simplex is a degeneration of $M_g$.

The series of $T_n$ is special for several reasons. To begin with,
starting with $T_4$, which is the boundary complex of the tetrahedron,
$T_n$ is gotten from $T_{n-1}$ by starring a vertex into an edge and
there is a well defined rule for which edge to star in. This gives a
systematic way of generating the Stanley-Reisner degenerations of the
$V_{2g-2}$.

Secondly, for  $4 \le n \le 10$, the  $T_n$ are the boundary complexes of
the \emph{convex deltahedra}, i.e.\ we see all convex deltahedra except the
icosahedron. Recall that a deltahedron is a 3-dimensional polytope with regular
triangles as faces.  There are exactly 8 convex deltahedra as proven
in \cite{fw:on}. Drawings, names and descriptions may be found for
example in \cite[Figure 2.18]{cr:po}. 

This seems at the moment to be just a nice
coincidence but there might be deeper explanation. Although the $T_n$
come in a series there is (as usual) no system relating their
automorphism groups $\Aut(T_n)$. Yet one can check case by case that
for $6 \le g \le 9$,  if $r_g$ is the index of the Fano homogeneous space
in the Mukai list, then $|\Aut(T_{g+1})| = 24 - 2r_g$.

Finally, for $4 \le n \le 10$, the Stanley-Reisner scheme of the cone
over $T_n$ is unobstructed. In fact the Stanley-Reisner ring of $T_n$
has trivial $T^2$. Thus $V_4'$, $V_6$, \ldots, $V_{16}$ all degenerate to
Fano Stanley-Reisner schemes which are smooth points in the
relevant Hilbert schemes. This can be used to find toric varieties to which the
Fano threefolds degenerate (Proposition
\ref{cor:toricdegen}). The point is
that, if a Stanley-Reisner scheme to which we degenerate
corresponds to a smooth point in some Hilbert scheme, any toric
variety also degenerating to this Stanley-Reisner scheme must
deform to a variety corresponding to a general point on the same
component of the Hilbert scheme. Our original motivation for this
article was in fact to find such toric degenerations, which have
become of interest in connection with mirror symmetry, see for example
\cite{pr:we} and \cite{ilp:to}.

These results point towards at least two continuations. One can ask if
degenerations to Stanley-Reisner schemes help find toric
degenerations of other Fano threefolds. This is the subject of a
separate paper \cite{ci:hil} where we, for $d \le 12$, study the Hilbert scheme of
degree $d$ smooth Fano threefolds in their anticanonical
embeddings. We use this to classify all possible degenerations of
these varieties to canonical Gorenstein toric Fanos. 

Taking another direction, one could ask for higher dimensional
combinatorial spheres with trivial $T^2$. This is the subject of the
second part of this paper. It is based on the observation that our
$T_9$ is boundary complex of the triaugmented triangular prism, which
again is the dual polytope of the 2-dimensional associahedron. Let
$\A_n$ be the  boundary
complex of the dual of the $(n-4)$-dimensional associahedron and $A_n$
its Stanley-Reisner ring. The 
main result of the second part of this paper is Theorem \ref{t2a} 
which states that
$T^2_{A_n} = 0$ for all $n$. For the sake of completeness we also
compute $T^1_{A_n}$ and describe the versal deformation of $\Proj (A_n
\otimes_k k[x_0, \dots , x_m])$. 

The $T_n$ in dimension two appear as edge starrings and unstarrings
of $\A_6$. In the last section we generalize this and use edge
starrings and unstarrings of $\A_n$ to find many more
combinatorial spheres with trivial $T^2$. Our Corollary \ref{hypero}
shows that if    $r_1, \dots ,r_m$ are integers with
$n > r_i \ge 4$ and
$$\sum_{i=1}^m r_i = n + 3(m-1)$$
then $\A_n$ is as a stellar subdivision of $\A_{r_1} \ast \A_{r_2}
\ast \dots \ast A_{r_m}$ via (different) series of edge
starrings. This yields many intermediate $(n-4)$-spheres whose
Stanley-Reisner ring has trivial $T^2$, generalizing the sequence
$T_6, \dots ,T_9$.

 In dimension 2 there is
exactly one edge starring of $\A_6$ yielding a sphere with $T^2=0$,
namely $T_{10}$ and any edge starring of $T_{10}$ has non-trivial
$T^2$. We finish this paper by listing all 74 combinatorial 
3-spheres with trivial $T^2$ coming from successive edge starrings of
$\A_7$.

Several results needed to prove $T^2_{A_n} = 0$ are valid in general
for flag complexes and we include them in a separate Section
\ref{flag}. To ensure that general linear sections $G(2,n)$ correspond
to generic points on Hilbert scheme components (needed for Corollary
\ref{g2ndegen}) we prove some results on deformations of
complete intersections in rigid Fano varieties which may be of general
interest, Proposition \ref{forget} and Corollary \ref{homspace}.

\begin{ack} We are grateful to 
Kristian Ranestad for helpful discussions. Much
of this work was done while the second author was visiting the
University of Oslo funded by ``sm{\aa }forsk-midler''. 
\end{ack}

 \section{Preliminaries} 

\subsection{Simplicial complexes and Stanley-Reisner Schemes}
\label{sec:sr} We now recall some basic facts about simplicial
complexes and Stanley-Reisner schemes, see for example \cite{sta:com}.
Let $[n]$ be the set $\{0,\ldots,n\}$ and $\Delta_n$ be the full
simplex $2^{[n]}$. An abstract \emph{simplicial complex} is any subset
$\K\subset\Delta_n$ such that if $f\in\K$ and $g\subset f$, then $g\in
\K$. Elements $f\in\K$ are called \emph{faces}; the dimension of a
face $f$ is $\dim f:=\#f-1$. Zero-dimensional faces are called
\emph{vertices} and we denote the set of vertices by
$V(\K)$. One-dimensional faces are called \emph{edges}. By
$\Delta_{-1}$ we will denote the simplicial complex consisting solely
of the empty set. Two simplicial complexes are isomorphic if there is
a bijection of the vertices inducing a bijection of all faces. We will
not differentiate between isomorphic complexes.

Given two simplicial complexes $\K$ and $\mcL$, their \emph{join} is
the simplicial complex
$$
\K * \mcL=\{f\vee g\ | \ f\in\K,\ g\in\mcL\}.
$$
 If $f\in \mathcal{K}$ is a face, we
may define
\begin{itemize}
\item the {\it link} of $f$ in $\mathcal{K}$;
$\;\link(f,\mathcal{K}):=\{g\in \mathcal{K}: g\cap f =
\emptyset \text{ and } g\cup f\in \mathcal{K}\}$,

\item the {\it open star} of $f$ in $\mathcal{K}$; 
$\;\st(f,\mathcal{K}):=\{g\in \mathcal{K}: f\subseteq g\}$, and

\item the {\it closed star} of $f$ in $\mathcal{K}$; 
$\;\overline{\st}(f,\mathcal{K}):=\{g\in \mathcal{K}: g\cup f\in
\mathcal{K}\}$.

\end{itemize}
Notice that the closed star is the subcomplex
$\overline{\st}(f,\mathcal{K}) = \bar{f}\ast \link(f,\mathcal{K})$.
If $f$ is an $r$-dimensional face of $\mathcal{K}$, define the {\em
valency} of $f$, $\nu(f)$, to be the number of $(r+1)$-dimensional
faces containing $f$.  Thus $\nu(f)$ equals the number of vertices in
$\link(f,\mathcal{K})$.

The {\em geometric realization} of $\mathcal{K}$, denoted
$|\mathcal{K}|$, is defined as
\[ |\mathcal{K}| = \big\{\alpha: [n] \to [0,1] | \{i | \alpha(i) \ne
0\}\in \mathcal{K} \mbox{ and $\,\sum_i \alpha(i) = 1$} \big\}\, .
\] In this paper we will be interested in the cases where $\K$ is a
combinatorial sphere or ball. A \emph{combinatorial $n$-sphere} is a
simplicial complex for which $|\mathcal{K}|$ is $PL$-homeomorphic to
the boundary of $\Delta_{n+1}$.  A \emph{combinatorial $n$-ball} is a
simplicial complex for which $|\mathcal{K}|$ is $PL$-homeomorphic to
$\Delta_{n}$. In general a simplicial complex $\mathcal{K}$ is a {\em
combinatorial $n$-manifold} (with boundary) if for all non-empty faces
$f\in \mathcal{K}$, $|\link(f,\mathcal{K})|$ is a combinatorial sphere
(or ball) of dimension $n-\dim f -1$.

 If $b \subseteq V(\K)$, denote by $\bar{b}$ the full simplex which is the
power set of $b$ and $\partial b = \bar{b} \setminus \{b\}$ its
boundary. 
We recall the notion of {\it stellar exchange} defined in
\cite{pac:plh}. (See also \cite{vir:lec}.)  Assume $\mathcal{K}$ is a
complex with a non-empty face $a$ such that $\link(a,\mathcal{K})=
\partial b \ast L$ for some non-empty set $b$ and $b$ is not a face of
$\link(a,\mathcal{K})$. We can now make a new complex
$\flip_{a,b}(\mathcal{K})$ by removing $\overline{\st}(a)=
\partial b \ast \bar{a} \ast L$ and replacing it with $
\partial a \ast \bar{b} \ast L$,
$$\flip_{a,b}(\mathcal{K}):=
(\mathcal{K}\setminus (
\partial b \ast \bar{a} \ast L)) \cup
\partial a \ast \bar{b} \ast L \, .$$ 

If $|b| = 1$, that is if $b$ is
a new vertex $v$, then the procedure $\flip_{f,v}(\K)$ is classically
known as {\em starring $v$ at the face $f$} and we denote the result
as $\sta(f, \K)$. (If $f$ also is a vertex we are just renaming $f$ with
$v$.) A complex $\K^\prime$ is known as a {\em stellar subdivision} of $\K$
if there exists a series $\K=\K_0, \K_1, \dots , \K_r = \K^\prime$
such that $\K_i = \sta(f, \K_{i-1})$ for some face $f \in K_{i-1}$.

To any simplicial complex $\K\subset\Delta_n$, we associate a
square-free monomial ideal $I_\K\subset \CC[x_0,\ldots,x_n]$
$$
I_\K:=\langle x_p \ | \ p\in\Delta_n\setminus\K\rangle
$$
where for $p\in\Delta_n$, $x_p:=\prod_{i\in p}x_i$. This gives rise to
the \emph{Stanley-Reisner ring} $A_\K:=\CC[x_0,\ldots,x_n]/I_\K$ and a
corresponding projective scheme $\PP(\K):=\Proj A_\K$ which we call a
Stanley-Reisner scheme.  The scheme $X:=\PP(\K)$ ``looks'' like the
complex $\K$: each face $f\in\K$ corresponds to some $\PP^{\dim
f}\subset X$ and the intersection relations among these projective
spaces are identical to those of the faces of $\K$. In particular,
facets of $\K$ correspond to the irreducible components of $X$. If
$\K$ is pure dimensional then the
degree of  $\PP(\K)$ will be the number of facets of $\K$. We also have
$$H^p({\mathbb
P}(\mathcal{K}),\mathcal{O}_{{\mathbb P}(\mathcal{K})}) \simeq
H^{p}(\mathcal{K};\mathbb{C})\, ,$$ by a result of Hochster, see
\cite[Theorem 2.2]{ac:def}. 

We also can make the affine scheme ${\mathbb A}(\mathcal{K})=\Spec
A_\mathcal{K}$. If $f$ is a subset of $V(\K)$, let $D_{+}(x_f)
\subseteq {\mathbb P}(\mathcal{K})$ be the chart corresponding to
homogeneous localization of $A_\mathcal{K}$ by the powers of
$x_f$. Then $D_{+}(x_f)$ is empty unless $f\in \mathcal{K}$ and if
$f\in \mathcal{K}$ then $$ D_{+}(x_f) = {\mathbb
A}(\link(f,\mathcal{K})) \times (k^{*})^{\dim f} \, .$$

If $\mathcal{K}$ is an \emph{orientable}
combinatorial manifold without boundary then the canonical sheaf is
trivial (\cite[Theorem 6.1]{be:gra}). Thus a smoothing of such a
${\mathbb P}(\mathcal{K})$ would yield smooth schemes with trivial
canonical bundle and structure sheaf cohomology equaling
$H^{p}(\mathcal{K};\mathbb{C})$. In particular if $\mathcal{K}$ is a
combinatorial sphere then a smoothing of ${\mathbb P}(\mathcal{K})$,
if such exists,  is
Calabi-Yau. We shall see that certain balls correspond in this way to
Fano schemes.

The combinatorial nature of Stanley-Reisner schemes also makes their
deformation theory more accessible than usual and has been studied
in \cite{ac:cot} and \cite{ac:def}. We will apply results
from these papers throughout.

\subsection{Cotangent cohomology of Stanley-Reisner schemes} 
We recall
one of the descriptions in \cite{ac:cot} of the multi-graded
pieces of $T^i_{A_\mathcal{K}}$ for any simplicial complex
$\mathcal{K}$. We refer also to this paper, \cite{ac:def} and the
references therein to standard works for
definitions of the various cotangent cohomology spaces. 

We recall first some
geometric constructions on simplicial complexes. To every non-empty
$f\in \mathcal{K}$, one assigns the {\em 
relatively open} simplex $\langle f\rangle \subseteq |\mathcal{K}|$;
\[ \langle f\rangle = \{\alpha\in |\mathcal{K}| \, | \, \alpha(i) \ne
0 \text{ if and only if } i\in f \}\, .
\] On the other hand, each subset $Y \subseteq \mathcal{K}$, i.e. $Y$
is not necessarily a subcomplex, determines a topological space
\[ \langle Y\rangle:=
\begin{cases} \bigcup_{f\in Y}\langle f\rangle& \text{if
$\emptyset\not\in Y$}, \\ \cone \left(\bigcup_{f\in Y}\langle
f\rangle\right)& \text{if $\emptyset\in Y$}\, .
\end{cases}
\] In particular, $\langle \mathcal{K}\setminus\{\emptyset\}\rangle =
|\mathcal{K}|$ and $\langle \mathcal{K}\rangle = |\cone(\mathcal{K})|$
where $\cone(\mathcal{K})$ is the simplicial complex $\Delta_0\ast
\mathcal{K}$.

Define
\begin{gather*} U_{b} = U_{b}(\mathcal{K}) :=\{f \in \mathcal{K} :
f\cup b \not\in \mathcal{K}\} \\ \widetilde{U}_{b} =
\widetilde{U}_{b}(\mathcal{K}) := \{f \in \mathcal{K} : (f\cup
b)\setminus \{v\} \not\in \mathcal{K} \text{ for some } v\in b\}
\subseteq U_{b}\,.
\end{gather*} Notice that $U_b=\widetilde{U}_{b}=\mathcal{K}$ unless $
\partial b$ is a subcomplex of $\mathcal{K}$. If $b \notin \K$ and
$\partial b \subseteq \K$ or $b \in \K$, define
$$L_b = L_b(\K):=
\bigcap_{b^\prime\subset b}\link(b^\prime, \K) \, .$$ We have
\begin{equation*} \mathcal{K}\setminus U_b =
\begin{cases} \emptyset \\ \overline{\st}(b)
\end{cases} \hspace{0.0em}\mbox{and}\hspace{0.7em}
\mathcal{K}\setminus \widetilde{U}_{b} =
\begin{cases}
\partial b \ast L_b& \text{if $b$ is a non-face},\\ (\partial b \ast
L_b) \cup \overline{\st}(b)& \text{if $b$ is a face}.
\end{cases} \vspace{1ex}
\end{equation*}

\begin{theorem}\emph{(\cite[Theorem 13]{ac:cot})}\label{topopen} The
homogeneous pieces in degree $\mathbf{c}=\mathbf{a}-\mathbf{b} \in
\mathbb{Z}^{|V|}$ (with disjoint supports $a$ and $b$) of the
cotangent cohomology of the Stanley-Reisner ring $A_\mathcal{K}$
vanish unless $a\in \mathcal{K}$, $\mathbf{b} \in \{0,1\}^{|V|}$, $b
\subseteq V(\link(a,\K))$ and $b\ne \emptyset$.  If these conditions
are satisfied, we have isomorphisms
\[ T^i_{A_\mathcal{K},\mathbf{c}} \;\simeq\; H^{i-1}\big(\langle
U_{b}(\link(a,\mathcal{K}))\rangle, \, \langle
\widetilde{U}_{b}(\link(a,\mathcal{K}))\rangle,\,k\big) \;\text{ for }
i=1,2
\] unless $b$ consists of a single vertex.  If $b$ consists of only
one vertex, then the above formulae become true if we use the reduced
cohomology instead.
\end{theorem}

Since $T^i_{A_\mathcal{K},\mathbf{c}}$ depends only on the supports
$a$ and $b$ we may denote it $T^i_{a-b}(\mathcal{K})$. We will have
use for
\begin{proposition}\emph{(\cite[Proposition
11]{ac:cot})}\label{aempty} If $\,b\subseteq V(\link(a))$, then the
map $f\mapsto f\setminus a$ induces isomorphisms
$T^i_{a-b}(\mathcal{K}) \simeq T^i_{\emptyset -b}(\link(a,
\mathcal{K}))$ for $i=1,2$.
\end{proposition}

We include the following for lack of reference.
\begin{proposition} \label{tensor} If $A$ and $B$ are $k$-algebras
then there are exact (split) sequences
$$ 0 \to T^i_B \otimes_k A \to T^i_{A\otimes_k B} \to T^i_A
\otimes_k B \to0$$ of cotangent modules for all $i$.
\end{proposition}
\begin{proof} Consider a cocartesian diagram of rings
$$
\begin{CD} B @>>> R\\ @AA\alpha A @AAA\\ S @>\beta>> A
\end{CD}$$ with both $\alpha$ and $\beta$ flat. Then by standard
properties of the cotangent modules (see e.g.\ \cite{an:hom}), if $M$
is a $B$-module, $T^i(R/A; M \otimes_S A) \simeq T^i(B/S; M) \otimes
A$. The morphisms $k \to A \to A \otimes_k B$ yield the Zariski-Jacobi
sequence
$$ \dots \to T^i(A \otimes_k B/A; A \otimes_k B) \to T^i_{A\otimes_k
  B} \to T^i(A/k; A \otimes_k B) \to \dots \, .$$ 
Since $B$ is a free $k$-module $T^i(A/k; A \otimes_k B) \simeq  T^i_A
\otimes_k B$ and the isomorphism above yields $ T^i(A \otimes_k B/A; A
\otimes_k B) \simeq  T^i_B\otimes_k A$. Thus the sequence reads
$$ \dots \to T^i_B\otimes_k A \to T^i_{A\otimes_k B} \to T^i_A
\otimes_k B \to T^{i+1}_B \otimes_k A \to \dots \, .$$ Switching $A$
and $B$ gives a natural section to $T^i_{A\otimes_k B} \to T^i_A
\otimes_k B$, so the map is surjective and the result follows.
\end{proof}

\subsection{Dual associahedra} \label{ad}
By $\A_n$ we denote the $n-4$ dimensional simplicial complex which is
the boundary complex of the dual polytope of the associahedron. The
associahedron (also known as the Stasheff polytope) plays a role in
many fields and various generalizations and realizations have appeared
in the recent literature, see e.g.\ the introduction in \cite{hl:re}
and the references therein. For our purposes the description of $\A_n$ given by
Lee in \cite{le:ass} is the most useful.

Consider the $n$-gon and index the vertices in cyclical order by $i =
1, \dots , n$. Denote by $\delta_{ij}$ the diagonal between vertex $i$
and vertex $j$. The set of $\frac{1}{2} n (n-3)$ diagonals will be the
vertex set of $\A_n$, call it $V_n$. A set $\{\delta_{i_1j_1}, \dots ,
\delta_{i_rj_r}\}$ of $r+1$ diagonals is an $r$-face of $\A_n$ if they
do not cross, i.e.\ they partition the $n$-gon into a union of $r+2$
polygons. The facets of $\A_n$ correspond therefore to the
triangulations of the $n$-gon with $n$ vertices. The number of facets
is thus the Catalan number
$$c_{n-2} = \frac{1}{n-2}\binom{2(n-3)}{n-3}\, .$$
The automorphism group
of $\A_n$ is the dihedral group $D_n$ and the action is induced by the
natural action on the $n$-gon.

For small $n$ we have $\A_3 = \{\emptyset\}$, $\A_4 $ is two vertices
with no edge which we denote $S^0$, $\A_5$ is the boundary of the
pentagon and $\A_6$ is the boundary complex of the triaugmented
triangular prism.

\section{Degenerations to unobstructed Fano Stanley-Reisner \\
schemes}\label{sec:gen} 
We state and prove here general results we will
apply to special cases in this paper and in \cite{ci:hil}. Consider a
triangulated $n$-ball $B$. Since triangulations of spheres are
degenerate Calabi-Yau, one may ask under what conditions the boundary
complex corresponds to the anticanonical divisor of $\mathbb{P}(B)$.
\begin{proposition}\label{fanocomp}  Let $T$ be any combinatorial sphere. For $m\geq 0$
consider the variety $X=\PP(T*\Delta_m)$.  Then $\omega_X\cong 
\mathcal{O}_X(-m-1)$. In particular, $X$ is Fano and if $m = 0$ the
natural embedding is anticanonical.
\end{proposition}
\begin{proof} Note that $A_{T*{\Delta_m}} = A_T[x_0, \dots ,
x_m]$. The canonical module $\omega_{A_T}$ of the Stanley-Reisner ring
$A_T$ equals $A_T$ as graded module, see \cite[section 7]{sta:com}. By
e.g. \cite[21.11]{eis:com}, it follows that
$\omega_{A_{T*{\Delta_m}}}=A_{T*{\Delta_m}}(-m-1)$.
\end{proof}

We will refer to such simplicial complexes as \emph{Fano
complexes}. In this paper we will be mostly interested in the special
situation when $T^2$ of the Stanley-Reisner ring of the combinatorial
sphere vanishes.
\begin{proposition}\label{cor:nonobstructed} If $\K$ is a
combinatorial sphere with $T^2_{A_{\K}} = 0$, then for the Fano
scheme $\PP(\K*\Delta_m)$, the obstruction space
$T_{\PP(\K*\Delta_m)/\PP^{m+n}}^2$ for the local Hilbert functor
vanishes. In particular, $\PP(\K*\Delta_m )$ is represented by a
smooth point in the corresponding Hilbert scheme.
\end{proposition}
\begin{proof} Since $A_{\K*\Delta_m}$ is just the tensor product over
$\CC$ of $A_\K$ with a polynomial ring, $T_{A_{\K*\Delta_m}}^2$
vanishes as well, see Proposition~\ref{tensor}.  The claim then
follows from \cite[Proposition 5.4]{ac:def} which states among other
things that in this case $T_{\PP(\K*\Delta_m)/\PP^{m+n}}^2$ is the
degree $0$ part of $T_{A_{\K*\Delta_m}}^2$.
\end{proof}

We now turn our attention to degenerations of smooth Fano varieties to toric
varieties. Consider some lattice $M$ and some lattice polytope
$\nabla\subset M_\QQ$ in the associated $\QQ$-vector space. By
$\PP(\nabla)$ we denote the toric variety
$$
\PP(\nabla)=\Proj \CC[S_\nabla]
$$
where $S_\nabla$ is the semigroup in $M\times \ZZ$ generated by the
elements $(u,1)$, $u\in \nabla \cap M$. By Theorem 8.3 and Corollary
8.9 of \cite{st:gr}, square-free initial ideals of the toric ideal of
$\PP(\nabla)$ are exactly the Stanley-Reisner ideals of unimodular
regular triangulations of $\nabla$, see loc. cit. for definitions.

\begin{proposition}\label{cor:toricdegen} Let $V \subseteq \PP^N$ be a
smooth Fano variety which is the generic point on its component in the
Hilbert scheme of $\PP^N$. Let $\K$ be a combinatorial sphere with
$T^2_{A_{\K}} = 0$ and assume $V$ degenerates to $\PP(\K \ast
\Delta_m)$. If $\nabla$ is a lattice polytope having a unimodular
regular triangulation of the form $\K \ast \Delta_m$, then $V$
degenerates to $\PP(\nabla)$.
\end{proposition}
\begin{proof} Now $\PP(\nabla)$ degenerates to $\PP(\K*\Delta_m)$ and
$\PP(\K*\Delta_m)$ is represented by a smooth point on the
corresponding Hilbert scheme (Proposition
\ref{cor:nonobstructed}). Since $V$ is represented by a general point
on the same component, $\PP(\nabla)$ must deform to $V$.
\end{proof}

\begin{remark}The toric varieties $\PP(\nabla)$ appearing in the above
proposition are quite special, since they are unobstructed.
\end{remark}

Note that if $V$ is a smooth Fano variety with very ample
anticanonical divisor, $V$ is a smooth point on a single irreducible
component of the relevant Hilbert scheme, cf. \cite[Proposition
2.1]{ci:hil}. A generic point on that component will be a smooth Fano
variety, to which we may apply \ref{cor:toricdegen}.

In this paper the smooth Fano varieties that appear are linear
sections of rational homogeneous manifolds. Recall that a rational
homogeneous manifold is of the form $G/P$ for a complex semi-simple
Lie group $G$ and parabolic $P$. Rational homogeneous manifolds are
Fano and rigid, i.e.\ $H^1(\Theta) = 0$, \cite[Theorem VII]{bo:ho}, and we may
use this to show that general sections are generic points on their
component in the Hilbert scheme of $\PP^N$.

For schemes $X \subseteq V$ let $\Def_{X/V}$ be the functor of
embedded deformations of $X$ in $V$. The forgetful functor $\Def_{X/V}
\to \Def_X$ is smooth if $T^1_V(\mathcal{O}_X) = 0$ and if  $V$ is
smooth this is $H^1(X,(\Theta_V)_{|X})$. 

\begin{proposition}\label{forget} Let $V$ be a subvariety of $\PP^n$
such that $H^1(V,\Theta_V)=0$ and such that Serre duality holds with
dualizing sheaf $\CO_V(-i)$. Let $X$ be a general complete
intersection of $V$ defined by $r$ forms of degree $m_k$ on $\PP^N$
with $\sum_{k=1}^r m_k <i$. Then $H^1(X,(\Theta_V)_{|X})=0$.
\end{proposition}
\begin{proof} Let $\I_X$ be the ideal sheaf of $X$ in $V$ and consider
the exact sequence of sheaves
$$
0 \to \I_X \to \CO_V \to \CO_X \to 0.
$$
After tensoring with $\Theta_V$ and passing to the long exact sequence
of cohomology, we see that the vanishing of $H^1(X,(\Theta_V)_{|X})$
follows from the vanishing of $H^1(V,\Theta_V)$ (which we have by
assumption) and the vanishing of $H^2(V,\I_X\otimes\Theta_V)$. We now
show the vanishing of the latter.

Let $\mathcal{F} = \bigoplus_{k=1}^r \CO_V(-m_k)$. Since
$X\hookrightarrow V$ is a complete intersection, we have a resolution
of $\I_X$ by the Koszul complex
\begin{equation*}
      \begin{CD} 0 @>>> \bigwedge^r \mathcal{F} @>d_r>> \dotsm @>d_3>>
\bigwedge^2\mathcal{F} @>d_2>>\mathcal{F} @ >d_1>> \I_X @>>> 0
      \end{CD}
\end{equation*} which we can split into short exact sequences
\begin{equation*}
      \begin{CD} 0 @>>> \I_{j} @>>> \bigwedge^j\mathcal{F} @>>>
\I_{j-1} @>>> 0
      \end{CD}
\end{equation*} with $\I_0:=\I_X$ and $\I_j:=\ker d_j$. We show that
$H^p(V,\I_j\otimes\Theta_V)=0$ for $p>1$ by induction on $j$. Indeed,
$H^p(V,\I_r\otimes\Theta_V)=0$ since $\I_r=0$. Suppose now that
$H^p(V,\I_j\otimes\Theta_V)=0$ for some $j$ and all $p>1$. Then from
the long exact sequence of cohomology, we have
$$H^p(V,\I_{j-1}\otimes\Theta_V)\cong
H^p\left(V,\bigwedge^j\mathcal{F} \otimes\Theta_V\right).$$ But
$\bigwedge^j\mathcal{F}\otimes\Theta_V$ is a direct sum of vector
bundles of the form $\CO(-l)\otimes \Theta_V$ with $l < i$, and by
Serre Duality and Kodaira vanishing, we have
$$
H^p\left(V,\CO(-l)\otimes\Theta_V\right)\cong
H^{n-p}\left(V,\CO(l-i)\otimes\Omega_V\right)=0
$$
where $n$ is the dimension of $V$.
\end{proof}

\begin{corollary} \label{homspace} Let $V$ be a rational homogeneous
manifold embedded in $\PP^N$ such that $\omega_V = \CO_V(-i_V)$ where
$i_V$ is the Fano index of $V$. If $X$ is a smooth complete
intersection of $V$ defined by $r$ general forms of degree $m_k$ on $\PP^N$
with $\sum_{k=1}^r m_k <i_V$, then $X$ is Fano and a generic point on its
component in the Hilbert scheme of $\PP^N$.
\end{corollary}
\begin{proof} From Proposition \ref{forget} we know that $\Def_{X/V}
\to \Def_X$ is smooth. In particular every deformation of $X$ arises
from moving the linear section in $V$. A general section is therefore
a generic point on the Hilbert scheme component.
\end{proof}

\section{Mukai varieties and deltahedra}\label{sec:tri}

We describe a series of triangulated $2$-spheres constructed
by means of edge starring. Let $T_4$ be the boundary complex of the
tetrahedron and $T_5$ the boundary complex of the triangular
bipyramid. For any $6 \leq n \leq 10$, define $T_n$ inductively to be
$\sta(f,T_{n-1})$ for any edge $f\in T_{n-1}$ whose link consists of
two vertices of valency four. This uniquely determines $T_n$.

 For
$6\leq n \leq 10$, these are 
exactly the triangulated spheres where the only vertex valencies are
$4$ or $5$. The list of $T_n$, $4 \le n \le 10$ coincides
with the boundary complexes of the convex deltahedra with 10 or less
vertices. Our rule cannot be applied to $T_{10}$, but we define $T_{11}$ to be
$\sta(f,T_{10})$ for any edge $f\in T_{10}$ whose link has one
valency-four vertex. These triangulations are pictured in Figure
\ref{fig:triangulations} projected from a vertex at infinity.  The
edges in which we star a vertex are the dashed line segments.

\begin{figure}
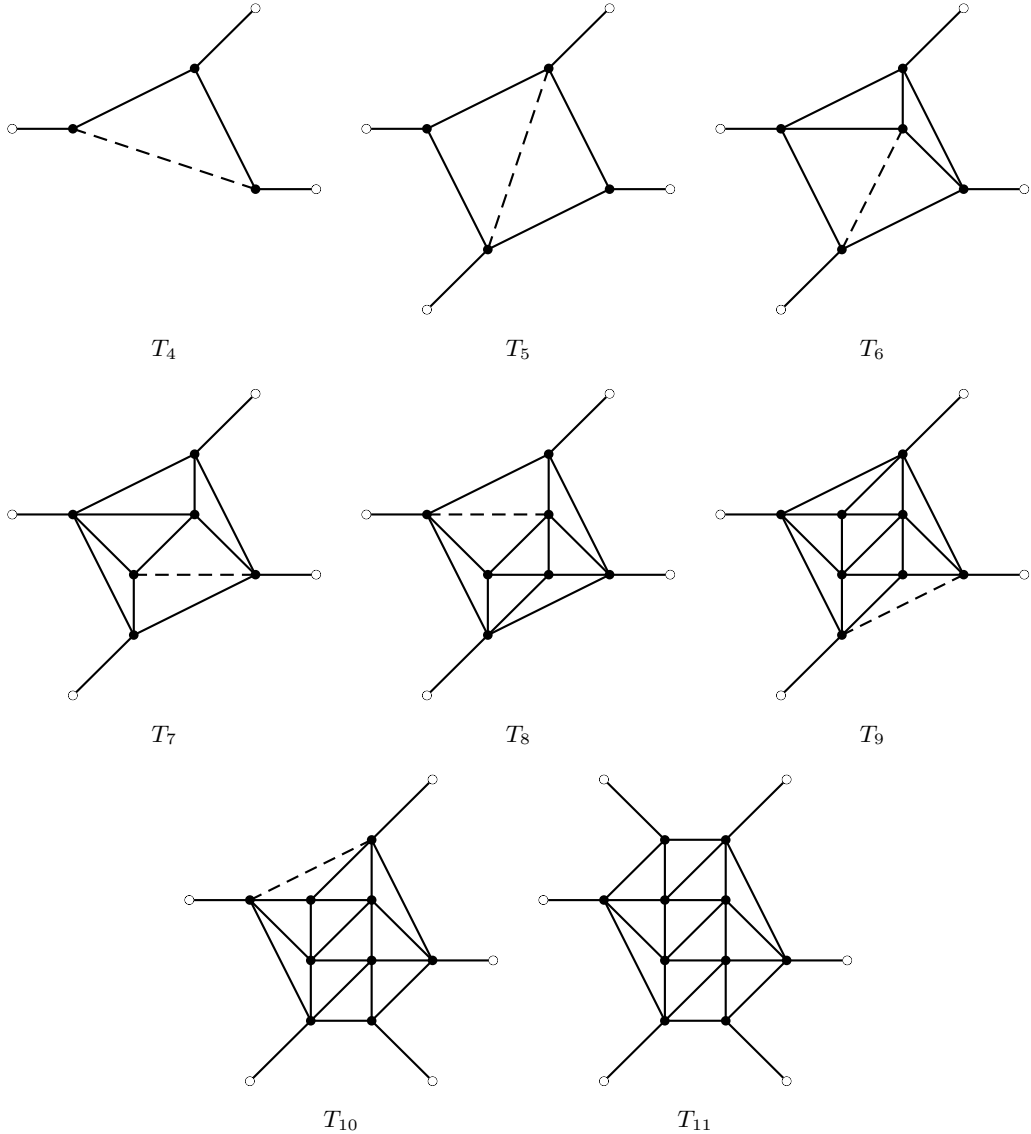
 
	\begin{center}
\subfloat[$T_{4}$]{\trifour}
\subfloat[$T_{5}$]{\trifive}
\subfloat[$T_{6}$]{\trisix}\\
\subfloat[$T_{7}$]{\triseven}
\subfloat[$T_{8}$]{\trieight}
\subfloat[$T_{9}$]{\trinine}\\
\subfloat[$T_{10}$]{\triten}
\subfloat[$T_{11}$]{\trieleven}
\end{center}	
\caption{Triangulations of the sphere
$T_n$}\label{fig:triangulations}
\end{figure}

The associahedra appear among the $T_n$. We have  $T_9 = \A_6$, the
octahedron boundary $T_6 = \A_4 \ast \A_4 \ast \A_4$ and the pentagonal
bipyramid boundary $T_7 = \A_4 \ast \A_5$.

When $4 \le n \le 10$ (the deltahedra case), $T_n$ is on the list of
triangulated $2$-spheres $T$ with $T_{A_{T}}^2=0$ classified in
\cite[Corollary 2.5]{iso:tor}. This may also be proven directly using
the results in Section \ref{flag}. For $\K = T_{11}$, \cite[Theorem
5.6]{ac:def} tells us that $\dim T_{A_\K, 0}^2=3$. 

\begin{theorem}\label{thm:degen} There is a flat degeneration of
$M_g$ to $\PP(T_{g+1}*\Delta_{i_g})$ for any $6\leq g\leq 10$ where
$i_g$ is one less than the Fano index of $M_g$, i.e. $i_6=i_{10}=2$,
$i_7=7$, $i_8=5$, and $i_9=3$.
\end{theorem}

We wish first to describe the method of proof. Given an ideal $I$ in a
polynomial algebra $P$ and a term order $\succ$, let $\init_\succ(I)$
be the initial ideal of $I$. There is a flat degeneration of $P/I$ to
$P/\init_\succ(I)$, see e.g. \cite[Chapter 15]{eis:com}, so we
want to find a term order such that $\init_\succ(I_{M_{g}})$ is the
Stanley-Reisner ideal of $T_{g+1}*\Delta_{i_g}$.

In \cite[Proposition 3.7.4]{st:alg} Sturmfels shows that there is a
term order for which the ideal of $G(2,n)$ in the Pl{\"u}cker
embedding has initial ideal equal to the Stanley-Reisner ideal of
$\A_n * \Delta_{n-1}$. He calls this order a \emph{circular order} and
variants of this circular order will be used throughout the proof.

Generators for $I_{M_{g}}$ may be found in the literature, we give
references in the proof. Assume we have found a term order such that
the ideal generated by the initial terms of these generators is the
Stanley-Reisner ideal of $T_{g+1}$. We may then invoke the following
useful result of Sturmfels and Zelevinsky.

Let now $I$ be a homogeneous ideal of degree $d$ in $P= \CC[x_0, \dots
, x_n]$ with $\dim P/I = r+1$ and $\mathcal{G} \subset I$ a finite
subset. Suppose the set $\{\init_\succ(g): g \in \mathcal{G}\}$
consists of \emph{square-free} monomials $x_{q_1}, \dots ,x_{q_s}$,
$q_i \subseteq [n]$.

\begin{proposition} [\protect{\cite[Proposition
7.3]{sz:max}}]\label{stzel} If all minimal (with respect to inclusion)
transversal subsets to $\{q_{1}, \dots , q_s\}$ have the same
cardinality $n-r$, and their number is less than or equal to $d$, then
$\mathcal{G}$ is a Gr{\"o}bner basis with respect to $\succ$.
\end{proposition}

To rephrase the result recall that a subset $p = \{i_1,\dots, i_k\}
\subseteq [n]$ is \emph{transversal} to $\{q_{1}, \dots , q_s\}$ if
there exists an injective map $f: \{1, \dots ,k\} \to \{1, \dots ,
s\}$ such that $i_j \in q_{f(j)}$. Clearly the minimal transversal
subsets to $\{q_{1}, \dots , q_s\}$ are in one to one correspondence
with the minimal prime ideals of $\langle x_{q_1}, \dots ,x_{q_s}
\rangle$.  Thus the proposition tells us that if $\langle x_{q_1},
\dots ,x_{q_s} \rangle$ is the Stanley-Reisner ideal of the simplicial
complex $\mathcal{K}$ with $n+1$ vertices, $\mathcal{K}$ is pure
$r$-dimensional, \emph{and} the number of facets of $\mathcal{K}$ is
less than or equal the degree of $I$, then $\mathcal{G}$ is a
Gr{\"o}bner basis. In particular $\langle x_{q_1}, \dots ,x_{q_s}
\rangle = \init_\succ(I)$.

\begin{proof}[Proof of Theorem \ref{thm:degen}] Clearly
$T_{g+1}*\Delta_{i_g}$ is pure dimensional and $\dim
T_{g+1}*\Delta_{i_g} = \dim M_g$. Moreover the degree of
$\PP(T_{g+1}*\Delta_{i_g})$ is the number of facets of $T_{g+1}$ which
is $2(g-1)$. (For any 2-sphere the number of facets is $2(\# \text{
vertices} -2)$ by the Euler formula.) This equals the degree of the
corresponding Fano 3-fold in $\mathbb{P}^{g+1}$. Our equations for the
Mukai varieties will be in $\mathbb{P}^{g+1 + i_g}$, so this will also
be the degree of the Mukai variety.

By the above remarks it is thus enough to give a set of generators
$\mathcal{G}$ and a term order $\succ$ such that $\{\init_\succ(g): g
\in \mathcal{G}\}$ are the generators of the Stanley-Reisner ideal of
$T_{g+1}$. We do this case by case.

\vspace{1ex}
\noindent {\bf Case $g=6$ and $g=8$: Grassmannians.}  We review the
argument in \cite[Proposition 3.7.4]{st:alg} for future
reference. Recall that the Grassmannian $G(2,n)$ is defined by the
ideal $I$ generated by the $4\times 4$ Pfaffians of an $n\times n$
antisymmetric matrix with coordinates
\begin{equation}\label{eqn:antisym} \left(\begin{array}{c c c c c}
0&x_{12}&x_{13}&\cdots&x_{1n}\\ -x_{12}&0&x_{23}&\cdots&x_{2n}\\
-x_{13}&-x_{23}&0&\cdots&x_{3n}\\ \vdots&\vdots&\vdots&\ddots&\vdots\\
-x_{1n}&-x_{2n}&-x_{3n}&\cdots&0
\end{array} \right)
\end{equation} A circular order $\prec_{\Circ}$ is any monomial order
which, for $1\leq i<j<k<l\leq n$, selects the monomial $x_{ik}x_{jl}$
as the lead term in the Pfaffian involving the rows and columns
$i,j,k,l$. Sturmfels showed that such terms orders exist and that the
Pfaffians form a Gr\"obner basis for them.  The initial ideal of $I$
is square-free, and corresponds to $\A_n * \Delta_{n-1}$. For $n=6$
this is the simplicial complex $T_9*\Delta_5$.

When $n=5$, $\A_5 * \Delta_{4}= C_5 * \Delta_4$, where $C_5$ is the
boundary of a pentagon. Now $M_6$ is defined by a general quadric in
$G(2,5)$. We can degenerate this quadric to $x_{\alpha}x_{\beta}$
where $x_{\alpha}, x_{\beta}$ do not appear in the monomials in the
initial ideal of the Pfaffians. The ideal generated by the initial
ideal of $I$ and this monomial is the ideal of $T_7*\Delta_2$.

\vspace{1ex}
\noindent {\bf Case $g=7$: $SO(5,10)$.} Equations for the orthogonal
Grassmannian $SO(5,10)$ can be found in
\cite{mu:cu} (see also \cite{rs:var}). Consider the polynomial ring $P$ in the variables
$u$, $x_{ij}$, and $y_k$ for $1\leq i < j \leq 5$, $1\leq k \leq
5$. Let $\Phi_{i}(x)$ denote the Pfaffian of the submatrix of
\eqref{eqn:antisym} for $n=5$ not involving the $i$th row and
column. Then the ideal of $SO(5,10)$ in $\PP^{15}$ is given by the
five equations of the form
$$
uy_i-(-1)^{i}\Phi_i(x)
$$ 
along with the five equations
\begin{equation*} \left(\begin{array}{c c c c c}
0&x_{12}&x_{13}&x_{14}&x_{15}\\ -x_{12}&0&x_{23}&x_{24}&x_{25}\\
-x_{13}&-x_{23}&0&x_{34}&x_{35}\\ -x_{14}&-x_{24}&-x_{34}&0&x_{45}\\
-x_{15}&-x_{25}&-x_{35}&-x_{45}&0
\end{array} \right)\cdot \left(\begin{array}{c} y_1\\ y_2\\ y_3\\
y_4\\ y_5
\end{array}\right)=0.
\end{equation*} Consider a circular monomial order on the variables
$x_{ij}$ as above, and expand this to any monomial order $\prec$ on
$P$ satisfying
$$u,y_2,y_3,y_4\prec y_1,y_5\prec x_{ij}.$$
Then the initial terms of the above ten equations are generators of
the ideal of $\PP(T_8*\Delta_7)$.

\vspace{1ex}
\noindent {\bf Case $g=9$: $LG(3,6)$.} Equations for the Lagrangian
Grassmannian $LG(3,6)$ can be found in \cite{ir:geo}. Consider
the polynomial ring $P$ in the variables $u$, $v$, $y_{ij}$, $z_{ij}$
for $1\leq i\leq j \leq 3$. Let $Y$ and $Z$ be the symmetric matrices
\begin{equation*} Y=\left(\begin{array}{c c c} y_{11}&y_{12}&y_{13}\\
y_{12}&y_{22}&y_{23}\\ y_{13}&y_{23}&y_{33}
\end{array}\right)\qquad Z=\left(\begin{array}{c c c}
z_{11}&z_{12}&z_{13}\\ z_{12}&z_{22}&z_{23}\\ z_{13}&z_{23}&z_{33}
\end{array}\right)
\end{equation*} and let $M_{i,j}(Y)$ respectively $M_{i,j}(Z)$ denote
the $(i,j)$th minor of $Y$ and $Z$. Then the ideal of $LG(3,6)$ in
$\PP^{13}$ is given by the 21 equations of the following form:
\begin{align*} (-1)^{i+j}M_{i,j}(Y)-vz_{ij}\qquad &1\leq i \leq j \leq
3\\ (-1)^{i+j}M_{i,j}(Z)-uy_{ij}\qquad &1\leq i \leq j \leq 3\\
Y_{i,\cdot}\cdot Z_{\cdot, i}-uv\qquad & 1\leq i \leq 3\\
Y_{i,\cdot}\cdot Z_{\cdot, j}\qquad & 1\leq i,j\leq 3,\quad i\neq j
\end{align*} Consider now any term order $\prec$ such that
$$
u,v,y_{13},z_{13}\prec y_{12},y_{23},z_{12},z_{23}\prec
y_{ii},z_{ii}\qquad i=1,2,3
$$
and the product of two monomials in the middle group is larger than
the product of a monomial from the right with a monomial from the
left. These conditions allow for freedom in the four comparisons
\begin{align*} y_{ii}z_{ij}\ \framebox{??}\ y_{ij}z_{jj}\qquad 0\leq
i,j\leq 3,\quad |i-j|=1.
\end{align*} Imposing any further conditions which resolve these four
comparisons completely determines the initial terms of the above 21
equations.

In fact, the 16 different possible ideals generated by these terms are
all Stanley-Reisner ideals coming from (different) triangulations of
the sphere with 10 vertices joined with $\Delta_3$; the triangulations
can be obtained by always choosing one of the two dotted diagonals in
each of the four quadrangles in Figure \ref{fig:lgflips}. Exactly two
of these triangulations are isomorphic to $T_{10}$. One possible way
to get $T_{10}$ is by imposing the additional condition
$x_{ii},y_{ii}\prec x_{jj},y_{jj}$ for $i<j$.

\begin{figure}
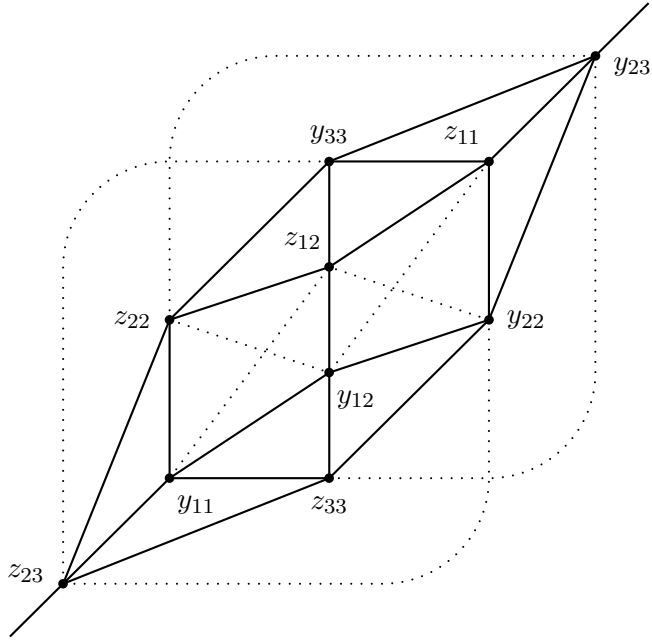

	\begin{center} \tritenflips
\end{center}
\caption{Triangulations coming from initial ideals for 
$LG(3,6)$}\label{fig:lgflips}
\end{figure}

\vspace{1ex}
\noindent {\bf Case $g=10$: $\GG_2$.} The $G_2$ Grassmannian can be
embedded in $G(2,7)$ as a linear section, see for example
\cite{ka:rel}. Let $P$ be the polynomial ring in variables
$r,u,w,x_{ij},y_{ij},z_i$ with $i,j\in\{0,1\}$. Then the ideal of
$\GG_2$ in $\PP^{13}$ is generated by the $4\times 4$ Pfaffians of the
matrix
\begin{equation*} \left(\begin{array}{c c c c c c c}
0&-x_{10}&x_{11}&w&y_{11}&y_{10}&u\\ x_{10}&0&-v&y_{00}&r&z_0&x_{00}\\
-x_{11}&v&0&y_{01}&z_1&-w-r&x_{01}\\
-w&-y_{00}&-y_{01}&0&x_{01}&-x_{00}&v\\
-y_{11}&-r&-z_1&-x_{01}&0&u&x_{11}\\
-y_{10}&-z_0&w+r&x_{00}&-u&0&x_{10}\\
-u&-x_{00}&-x_{01}&-v&-x_{11}&-x_{10}&0
\end{array}\right).
\end{equation*} Note that this is not a minimal generating set for the
ideal, it only needs 28 generators.

Consider any term order $\prec$ such that
$$
u,v\prec r,w,x_{ij}\prec y_{ij},z_i\qquad i,j\in\{0,1\}
$$
and the product of two monomials in the middle group is larger than
the product of a monomial from the right with a monomial from the
left. Similar to the $g=9$ case, these conditions allow for freedom in
the three comparisons
\begin{align*} x_{00}x_{11}\ \framebox{??}\ x_{01}x_{10}\\
x_{00}y_{01}\ \framebox{??}\ x_{01}y_{00}\\ x_{10}y_{11}\
\framebox{??}\ x_{11}y_{10}.
\end{align*} Imposing any further conditions which resolve these three
comparisons completely determines the ideals generated by initial
terms of the above 35 equations.

\begin{figure}
	\begin{center} \trielevenflips
\end{center}
\caption{Triangulations coming from initial ideals for
$\GG_2$}\label{fig:g2flips}
\end{figure}

The 8 different possible ideals generated by these terms are all
Stanley-Reisner ideals coming from (different) triangulations of the
sphere with 11 vertices joined with $\Delta_2$; the triangulations can
be obtained by always choosing one of the two dotted diagonals in each
of the three quadrangles in Figure \ref{fig:g2flips}. Exactly two of
these triangulations are isomorphic to $T_{11}$.
\end{proof}

For $-1\leq k \leq i_g-1$ let
$h_0,\ldots,h_{i_g-k-1}$ be general hyperplanes in $\mathbb{P}^{g+1 +
i_g}$. We can degenerate each $h_j$ to the coordinate $x_j$
corresponding to the $j$th vertex of $\Delta_{i_g}$. Combined with the
degeneration of $M_g$ in Theorem \ref{thm:degen}, this gives a flat
family with general fiber $M_g \cap \{h_0 = \dots = h_{i_g-k-1}= 0\}$
and special fiber $\PP(T_{g+1}*\Delta_k)$. We sum this up as

\begin{corollary}\label{cor:secdegen} Fix $6\leq g \leq 10$ and some
$-1\leq k \leq i_g-1$. Let $V$ be the intersection of $M_g$ with
$(i_g-k)$ general hyperplane sections. Then $V$ degenerates to
$\PP(T_{g+1}*\Delta_k)$.
\end{corollary}

When $3 \le g \le 5$ the $V_4'$, $V_6$, and $V_8$ are complete
intersections. Clearly they degenerate to the complete intersections
$\PP(T_{g+1}*\Delta_0)$. We get therefore Stanley-Reisner
degenerations of all rank one index one Fano threefolds of genus
$3\leq g \leq 10$.

\begin{remark} The boundary complex of the icosahedron, the last
deltahedron, gives a triangulation of the sphere with $12$ vertices
such that $T^2$ of the corresponding Stanley-Reisner ring vanishes. If
we call this complex $T_{12}$, there is no smooth Fano threefold which
has an embedded degeneration to $\PP( T_{12}*\Delta_0)$. Indeed, the
link of the vertex $\Delta_0$ corresponds to an affine chart $U_0
=\Spec A_{T_{12}}$. Since $T_{12}$ has no vertices of valency less
then 5, it follows easily from \cite[Theorem 4.6]{ac:def} that $U_0$
has no deformations in negative degree and is therefore not
smoothable.  This fits nicely with a ``missing'' Hilbert scheme
component. One computes, e.g. using Theorem \ref{topopen}, that $\PP(
T_{12}*\Delta_0)$ lies on a component of the Hilbert scheme with
dimension $174$. However, a component of the Hilbert scheme whose
general element is a smooth Fano must have dimension $173$, $175$,
$176$, or $177$ as can be computed from the classification in
\cite{mm:clm} by using \cite[Proposition 2.1]{ci:hil}.
\end{remark}

For $3\leq g \leq 9$ the above and Corollary \ref{homspace} show that
Proposition \ref{cor:toricdegen} applies so we get
\begin{corollary}\label{cor:toricdegenM} Let $V$ be a general
element in the deformations class $V_{2(g-1)}$ of rank one index
one smooth Fano threefolds of genus $3 \le g \le 9$. If $\nabla$ is a
lattice polytope having a unimodular regular triangulation of the form
$T_{g+1}*\Delta_0$, then $V$ degenerates to $\PP(\nabla)$.
\end{corollary}

\begin{remark} In the case $g = 10$ we know $T^2_{A_{T_{11}}} \ne
  0$. In fact one may compute that if $\K$ is the Fano complex $T_{11}
  \ast \Delta_0$ then $T^2_{\PP(\K)}$ is 6 dimensional. The Hilbert
  scheme locally at this scheme will consist of two components of
  dimensions 153 and 152. The rank one index one $V_{18}$ is on the
  153 dimensional component. This can be used to find toric
  degenerations of $V_{18}$. Indeed if $\nabla$ is a lattice polytope
  having a unimodular regular triangulation of the form $T_{11} \ast
  \Delta_0$ and $h^0(\PP(\nabla), N) = 153$ with
  $T^2_{\PP(\nabla)/\PP^{11}} = 0$, then $V_{18}$ degenerates to
  $\PP(\nabla)$. These two vector spaces can be computed explicitly
  via a comparison theorem, see \cite[Proposition 4.2]{ci:hil}.
\end{remark}

\section{Cotangent cohomology for flag complexes}\label{flag}
Recall that $\K$ is
called a \emph{flag complex} if any set of pairwise incident vertices
is a face. We may reformulate this as $b \subseteq V(\K)$, $b \notin
\K$ and $\partial b \subseteq \K$ implies $|b|=2$. Thus it is clear
that $\K$ is a flag complex if and only if $I_\K$ is generated by
quadratic monomials. For such a quadratic monomial generator
$x_{v}x_w$ we call the subset $\{v,w\}$ a \emph{non-edge}.  If $\K$ is
a flag complex, then so is $\link(f,\K)$ for all $f \in \K$. One
simple way to see this is to observe that the Stanley-Reisner ideal of
$\link(f)$ is gotten from $I_\K$ by putting $x_v = 1$ for all $v \in
f$.

A flag complex is determined by its edge graph $\Gamma =
\Gamma(\K)$, since $f \in \K$ if and only if the subgraph of $\Gamma$
induced by the vertices in $f$ is complete.  It is the \emph{clique
complex} of its edge graph. On the other hand the clique complex of
any simple graph is a flag complex.

When $f \in \K$ we always have $\link(f,\K) \subseteq L_f(\K)$, but
for a flag complex they are equal. In fact we have
\begin{lemma} \label{bin} A simplicial complex $\K$ is a flag complex
if and only if $\link(f,\K) = L_f(\K)$ for all faces $f$ with $\dim f
\ge 1$. 
\end{lemma}
\begin{proof} Assume first that $\K$ is a flag complex. If
$\link(f,\K) \ne L_f$ and $\dim f \ge 1$, there exists a non-empty
$g\in \K$ with $g \cup f^\prime \in \K$ for all faces $f^\prime\subset
f$ and $g \cap f^\prime = \emptyset$ for all faces $f^\prime\subset
f$, but $g \notin \link(f)$. Clearly $f \cap g = \emptyset$, so $f
\cup g \notin \K$. If $g^\prime \subset g$ then $g^\prime \in L_f$, so
we may choose $g$ minimal, i.e.\ we may assume $g^\prime \cup f \in
\K$ for all $g^\prime \subset g$. But then, if $b = g \cup f$, $b
\notin \K$ and $\partial b \subseteq \K$. Since $\K$ is a flag complex
we must have $|g \cup f| = 2$ contradicting $\dim f \ge 1$.

Assume now $\link(f,\K) = L_f(\K)$ for all faces $f$ with $\dim f \ge
1$. If $|b| \ge 2$, $b \notin \K$ and $\partial b \subseteq \K$ let
$f$ be a facet of $\partial b$ and $v = b \setminus f$. If $f^\prime
\subset f$ then clearly $v \cap f^\prime = \emptyset$. Moreover
$f^\prime \cup v$ will be in some other facet of $\partial b$, so $v
\in L_f$. On the other hand $v \cup f = b$ so $v \notin
\link(f)$. Therefore $f$ must be a vertex and $|b| = 2$.
\end{proof}

\begin{lemma} \label{bface} If $\K$ is a flag complex and $b \in \K$
and $|b| \ge 2$ then $T^i_{\emptyset-b}(\K) = 0$ for $i=1,2$.
\end{lemma}
\begin{proof} From Lemma \ref{bin} we know $\link(b,\K) = L_b$. Recall
that $\overline{\st}(b) = \overline{b} \ast \link(b)$. Thus $(\partial
b \ast L_b) \cup \overline{\st}(b) = (\partial b \ast \link(b)) \cup
\overline{\st}(b) = \overline{\st}(b)$. It follows that
$\widetilde{U}_b = U_b $ so $T^i_{\emptyset-b}(\K) = 0$ by Theorem
\ref{topopen}.
\end{proof}

\begin{remark} For $T^1$ the above is a rather trivial observation
since the ideal is generated by quadrics, but for $T^2$ there does not
seem to be an easy alternative argument.
\end{remark}

Since links of faces in flag complexes are flag complexes, Proposition
\ref{aempty} tells us that if we know $T^2_{\emptyset-b}(\K)$ for flag
complexes we know all $T^2_{a-b}(\K)$.
\begin{proposition} \label{t2flag} If $\K$ is a flag complex then
$T^2_{\emptyset-b}(\K) = 0$ unless
\begin{list}{\textup{(\roman{temp})}}{\usecounter{temp}}
\item $b = \{v\}$ is a vertex, then $T^2_{\emptyset-\{v\}}(\K) \simeq
{H}^{1}(|\mathcal{K}|\setminus | \overline{\st}(\{v\})|,k) $ or
\item $b$ is a non-edge, then $T^2_{\emptyset-b}(\K) \simeq
\widetilde{H}^{0}(|\mathcal{K}|\setminus |
\partial b \ast L_{b}|,k)$.
\end{list} In particular if $|K|$ is a sphere and $|L_b|$ is
contractible  then $T^2_{\emptyset-b}(\K) = 0$.
\end{proposition}
\begin{proof} Since $U_b=\widetilde{U}_{b}$ unless $\partial b
\subseteq \mathcal{K}$, it follows from Lemma \ref{bface} and Theorem
\ref{topopen} that $T^2_{\emptyset - b}(\K) = 0$ unless $b$ is a
vertex or non-edge.  The isomorphisms are true for all simplicial
complexes.  By Theorem~\ref{topopen} we have $T^2_{\emptyset-b} \simeq
H^{1}(\langle U_{b}\rangle, \langle \widetilde{U}_{b}\rangle)$.  If
$b\not\in \mathcal{K}$, then $\emptyset \in U_{b}$, so $\langle
U_{b}\rangle$ is a cone and $H^{1}(\langle U_{b}\rangle, \langle
\widetilde{U}_{b}\rangle)\simeq \widetilde{H}^{0}( \langle
\widetilde{U}_{b}\rangle) \simeq
\widetilde{H}^{0}(|\mathcal{K}|\setminus |
\partial b \ast L_{b}|,k)$.  If $b$ is a vertex, then
$\widetilde{U}_{b}=\emptyset$.
\end{proof}

\section{Cotangent cohomology for the dual associahedron} We will now
apply this to $\A_n$.  Let $A_n$ be the Stanley-Reisner ring of
$\A_n$. The simplicial complex $\A_n$ is a flag complex and the
non-edges consist of two crossing diagonals. The Stanley-Reisner ideal
of $\A_n$ is thus generated by the $\binom{n}{4}$ quadratic monomials
$x_{ik}x_{jl}$ with $ 1 \le i < j < k <l \le n$ in $k[x_{ij} : i < j ,
\, \delta_{ij} \in V_n] $. For a face $f$ let $\mathcal{P}_f$ be the
set of polygons in the partition of the $n$-gon defined by $f$.

\begin{lemma} \label{link} If $f \in \A_n$ has dimension $r$ and
splits the $n$-gon into $n_i$-gons, $i = 0, \dots , r+1$, then $\sum
n_i = n + 2(r+1)$ and $\link(f,\A_n) \simeq \A_{n_0} \ast \A_{n_1}
\ast \dots \ast \A_{n_{r+1}}$.
\end{lemma}
\begin{proof} The facets of $\link(f)$ may be seen by taking a
triangulation of the $n$-gon containing all the diagonals in $f$ and
then removing the diagonals in $f$. This
clearly gives the splitting.
\end{proof}

Let $i, j, k, l$ be labels of vertices on the $n$-gon with $i < j < k
<l$. Consider the inscribed quadrangle $Q = Q_{ijkl}$ with vertices
$\{i,j,k,l\}$ (see Figure \ref{8gon}). If $i+1 \le j-1$ then $\delta_{ij}$ is a diagonal
splitting the $n$-gon into two polygons. Let $\A_{ij}$ be the dual
associahedron corresponding to the polygon with vertices $\{i, i+1,
\dots , j-1, j\}$, i.e.\ having the common edge $\delta_{ij}$ with
$Q$. Finally let $B_{ij}$ be the triangulated ball $B_{ij} =
\{\delta_{ij}\} \ast \A_{ij} \subseteq \A_n$. If $j = i+1$ set
$B_{ij}$ to be the empty complex. Now do the same for the other edges
of $Q$ to get the 4 pairwise disjoint sub-complexes $B_{ij}, B_{jk},
B_{kl}, B_{li}$ and set $B_{ijkl} = B_{ij} \ast B_{jk} \ast B_{kl}
\ast B_{li}$.

\begin{figure}
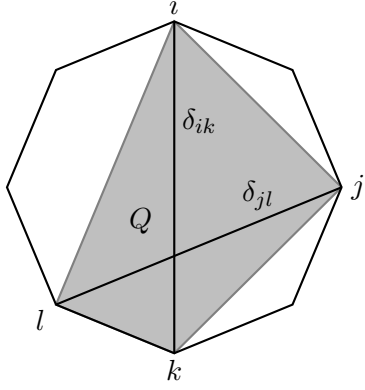
 \centering
\octa
	\caption{Two crossing diagonals and the quadrangle $Q$.}
\label{8gon}
\end{figure}

\begin{lemma} \label{boundb} If $b$ is the non-edge consisting of two
crossing diagonals $\delta_{ik}, \delta_{jl}$ with $i < j < k <l$,
then $L_b(\A_n) = B_{ijkl}$. In particular, if $n \ge 5$, $|L_b|$ is a
$(n-5)$-dimensional ball.
\end{lemma}
\begin{proof} It is clear that $ B_{ij} \ast B_{jk} \ast B_{kl}
\ast B_{li} \subseteq L_b$. Assume $f \notin B_{ij} \ast B_{jk} \ast
B_{kl} \ast B_{li}$. Then $f$ must contain a diagonal $\delta$ which
is either inside $Q$ or crosses one of the edges of $Q$. In the first
case $\delta$ must be either $\delta_{ik}$ or $\delta_{jl}$ and can
therefore not be in the corresponding link. In the second case
$\delta$ must also cross at least one of $\delta_{ik}, \delta_{jl}$.

The space $|L_b|$ is a ball since the join of two balls is a
ball. Note that $\dim B_{ij} = \dim \A_{ij} + 1$. The dimension of
$L_b$ is $\dim B_{ij} +\dim B_{jk} + \dim B_{kl} + \dim B_{li} + 3 = n-5$. 
\end{proof}

\begin{theorem} \label{t2a} The module $T^2_{A_n} = 0$ for all $n$.
\end{theorem}
\begin{proof} We will use induction on $n$. For $n = 4$ the result is
clear since $A_4 = k[x,y]/(xy)$. We must show that all the spaces
$T^2_{a-b}(\A_n)$ vanish. If $a \ne \emptyset$ we may use Proposition
\ref{aempty}. Note that $A_{\K \ast \mathcal{L}} = A_{\K} \otimes_k
A_{\mathcal{L}}$. Thus if $T^2_{A_k} = 0$ for all $k < n$, we get
$T^2_{a-b}(\A_n) = 0$ for $a \ne \emptyset$ by Lemma \ref{link} and
Proposition \ref{tensor}. We are left with the case $a = \emptyset$
and this follows directly from Proposition \ref{t2flag} and Lemma
\ref{boundb}.
\end{proof}

In relation to degenerations to toric varieties Proposition \ref
{cor:toricdegen}, Corollary \ref{homspace} and the previously referred
to \cite[Proposition 3.7.4]{st:alg} imply
\begin{corollary}\label{g2ndegen} Let $\nabla$ is a lattice polytope
having a unimodular regular triangulation of the form $\A_n \ast
\Delta_m$, $0 \le m \le n-1$, then $\PP(\nabla)$ is a degeneration of
a codimension $n-m-1$ linear section of $G(2,n)$.
\end{corollary}

\begin{remark} Note that Corollary \ref{g2ndegen} identifies toric
degenerations which do not arise via the standard method of finding
toric initial ideals of the Pl\"ucker ideal.  Consider the simple
example of
$$G(2,4)=V(x_{12}x_{34}-x_{13}x_{24}+x_{14}x_{23})\subset\PP^5.$$
Clearly $G(2,4)$ degenerates to $X=V(x_{12}x_{34}-x_{13}^2)$, but this
does not correspond to an initial ideal of $G(2,4)$, as say
$V(x_{12}x_{34}-x_{13}x_{24})$ does. Nonetheless, we can see this
degeneration with our methods: the moment polytope of $X$ has a
regular unimodular triangulation of the form $\A_4*\Delta_3$.
\end{remark}

We proceed to compute the $T^1_{a-b}(\A_n)$. If $b = \{\delta_{ik},
\delta_{jl} \}$ consists of two crossing diagonals set $Q_b$ to be the
corresponding inscribed quadrangle $Q_{ijkl}$.

\begin{theorem} \label{t1a} The structure of $T^1_{A_n}$ is given by
\begin{list}{\textup{(\roman{temp})}}{\usecounter{temp}}
\item For $a,b \subseteq V(\A_n)$, if $a \in \A_n$, $b \subseteq
V(\link(a))$ and $b$ consists of two crossing diagonals with $Q_b \in
\mathcal{P}_a$ then $\dim_k T^1_{a-b}(\A_n) = 1$, otherwise
$T^1_{a-b}(\A_n) =0$.
\item There is a one-to-one correspondence between inscribed
quadrangles in the $n$-gon and a minimal set of generators for the
$A_n$-module $T^1_{A_n}$.
\end{list}
\end{theorem}
\begin{proof} That $\dim T^1_{a-b} \in\{0,1\}$ is a general fact for
combinatorial manifolds (without boundary), see (\cite[Lemma
4.3]{ac:def}). Moreover $T^1_{a-b} = 0$ if $b$ is a vertex
(loc.cit.). Thus by Lemma \ref{bface}, Theorem \ref{topopen} and Lemma
\ref{boundb} we are left with the case $b \notin \A_n$ and $|b|=2$.

Assume first that $a = \emptyset$. For a combinatorial manifold $\K$,
\cite[Theorem 4.6]{ac:def} says $T^1_{\emptyset-b}(\K) \ne 0$ iff $\K
= L_b \ast \partial b$ i.e.\ the suspension of $L_b$. But if $n \ge
5$, Lemma \ref{boundb} tells us that $|L_b|$ is a ball, so this is
impossible unless $n=4$. Indeed we do have $\A_4 = \{\emptyset \}
\ast \partial b$ where $b$ consists of the two diagonals.

If $a \ne \emptyset$, Proposition \ref{aempty} and Lemma \ref{link}
tell us that $T^1_{a-b} \simeq T^1_{\emptyset -b}(\A_{n_0} \ast
\A_{n_1} \ast \dots \ast \A_{n_{r+1}})$ where the product is over the
polygons in $\mathcal{P}_a = \{P_0, \dots , P_{r+1} \}$. Now diagonals
in different $P_i$ will not cross so $b \subseteq V(\A_{n_i})$ for one
$i$ which we may assume is $0$. This means that $L_b(\link(a)) =
L_b(\A_{n_0}) \ast \A_{n_1} \ast \dots \ast \A_{n_{r+1}}$, which is a
sphere iff $n_0 = 4$, i.e.\ $P_0 = Q_b$. On the other hand if $n_0 =
4$ then $\link(a) = \partial b \ast \A_{n_1} \ast \dots \ast
\A_{n_{r+1}} = \partial b \ast L_b$ so $T^1_{a-b} \ne 0$. This proves
(i).

To prove (ii) we may assume $n > 4$. Consider the function that takes
an inscribed quadrangle to the set $a(Q)$ consisting of diagonals
which are edges of $Q$. Note there could be 1,2,3 or 4 diagonals in
$a(Q)$ depending on the placement of $Q$. If $b(Q)$ is the set of
diagonals in $Q$ let $\mathbf{c}(Q) = \chi_{a(Q)} - \chi_{b(Q)} \in
\ZZ^{n(n-3)/2}$ where $\chi_A$ is the characteristic vector of the
subset $A$. Then $Q \mapsto \text{generator of } T^1_{\mathbf{c}(Q)}$
sets up the correspondence.

Indeed from (i) we know that $T^1_{\mathbf{c}} \ne 0$ means that
$\mathbf{c} = \mathbf{a} - \mathbf{b}$, with disjoint supports $a$ and
$b$, and $\mathbf{b} = \chi_{b(Q)}$ for some $Q$ with $Q \in
\mathcal{P}_a$. The last inclusion implies that $a(Q) \subseteq a$. An
element in the one-dimensional $T^1_{\mathbf{c}}$ equals $\lambda \,
x^{\mathbf{a} - \chi_a(Q)} \cdot (\text{generator associated to
$Q$})$, where $\lambda$ is some constant and clearly $x^{\mathbf{a} -
\chi_a(Q)} \ne 0$ in $A_n$.\end{proof}

For the sake of completeness we prove a result about deformation
spaces for the Stanley-Reisner scheme of $\mathcal{A}_n \ast
\Delta_m$. Let $y_0, \dots , y_m$ be the variables corresponding to
vertices of 
$\Delta_m$. Since $\dim T^1_{A_\mathcal{K},\mathbf{c}}$ is $0$ or $1$
we may represent a basis element by a rational monomial
$x^\mathbf{c}$. We are using cyclic indices on the $n$-gon, if e.g.\
$n=7$ then $j = 5, \dots , 2$ means $j \in \{5, 6, 7, 1, 2\}$.

Consider the sets of $T^1$elements
\begin{align}\bigg\{\frac{x_{ij}x_{kl}}{x_{i+1,j}x_{i,j-1}} &: i = 1, \dots
, n, \, j=i+3,\, \delta_{kl} \in \vertices (\link(\{\delta_{ij}\}))
\cup \{\delta_{i,j}\} \bigg\}\label{s1}\\
\bigg\{\frac{x_{ij}y_{k}}{x_{i+1,j}x_{i,j-1}} &: i = 1, \dots , n, \, j=i+
3, \, k = 0, \dots , m\bigg\} \label{s2}\\
\bigg\{\frac{x_{i,j-1}x_{i,j+1}}{x_{i,j}x_{j-1,j+1}} & : i = 1, \dots , n,
\, j = i+3, \dots , i-3\bigg\} \label{s3}\\
\bigg\{\frac{x_{ij}x_{i+1,j-1}}{x_{i+1,j}x_{i,j-1}} &: i = 1, \dots , n, \,
j = i+4, \dots , i-2\bigg\} \, . \label{s4}
\end{align} Let $\mathcal{B}_{n,m}$ be the union of these four sets.

\begin{figure}
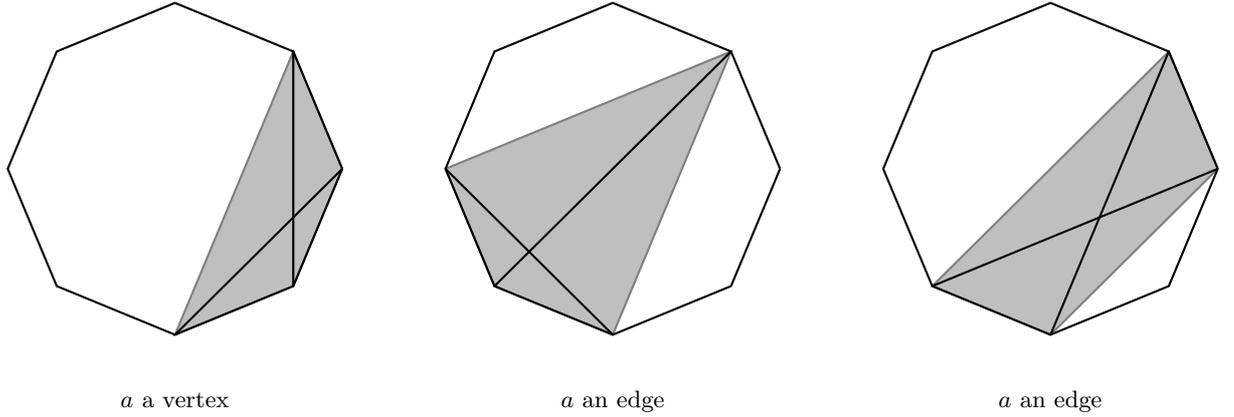
 \centering \subfloat[$a$ a
vertex]{\octb}
\subfloat[$a$ an
edge]{\octc} \subfloat[$a$
an edge]{\octd}
\caption{Three types of contributions to $T^1$.}
  \label{types}
\end{figure}

\begin{theorem} If $n\ge 5$ the versal base space of the
Stanley-Reisner scheme of $\mathcal{A}_n \ast \Delta_m$ is smooth of
dimension $\frac{1}{2}n(n^2-4n-3) + n(m+1)$ unless $n=6$ and
$m=-1$. The set $\mathcal{B}_{n,m}$ is a basis for the tangent
space. For the exceptional case $\A_6$ one must add $1$ to the formula
because of non-algebraic deformations.
\end{theorem}
\begin{proof} Let $X$ be the Stanley-Reisner scheme, $A$ the
Stanley-Reisner ring and $N = \frac{1}{2} n (n-3) + m$ so that $X
\subset \mathbb{P}^N$. The local Hilbert functor
$\Def_{X/\mathbb{P}^N}$ is unobstructed by Theorem \ref{t2a},
Proposition \ref{tensor} and \cite[Proposition 5.4]{ac:def}. The
forgetful map $\Def_{X/\mathbb{P}^N} \to \Def_{X}$ is smooth if
$H^1(\Theta_{\mathbb{P}^N} \otimes \mathcal{O}_X) = 0$. But by the
Euler sequence and a result of Hochster (see e.g.\ \cite[Theorem
2.2]{ac:def}) $H^1(\Theta_{\mathbb{P}^N} \otimes \mathcal{O}_X) \simeq
H^2( \mathcal{O}_X) \simeq H^2(\mathcal{A}_n \ast \Delta_m,k)$. This
group vanishes unless $m=-1$ and $n=6$, i.e.\ the K3 case. This case
is covered by Proposition \ref{a6} below. The dimension of the base
space is thus $\dim_k T^1_X$. Since $A$ is Cohen-Macaulay,
\cite[Theorem 3.9]{kle:def} tells us that $ T^1_X  \simeq {(T^1_A)}_0$
if $\dim X \ge 3$. 

Aside from the special case 
$\A_6$, the $X$ with $\dim X \le 2$ and $n\ge 5$ correspond to either
the boundary of the $5$-gon or the boundary of the $5$-gon $\ast
\Delta_0$. In these 2 cases the smoothness of the forgetful map and
\cite[Proposition 5.4]{ac:def} imply that $H^0(\mathcal{T}^1_X) \simeq
{(T^1_A)}_0$. On the other hand there is always an exact sequence
$$0 \to H^1(\Theta_X) \to T^1_X \to H^0(\mathcal{T}^1_X)  \to
H^2(\Theta_X) \, .$$ A computation in each case shows that
$H^1(\Theta_X) = H^2(\Theta_X)=0$ so we may compute $T^1_X$ as $
{(T^1_A)}_0$ directly or as below. For $\A_6$ (as for all
Stanley-Reisner K3 surfaces, see \cite[Theorem 5.6]{ac:def}) there is
a one dimensional contribution of non-algebraic deformations coming
from $H^1(\Theta)$.

To compute $ {(T^1_A)}_0$ we must consider $T^1_{a -b}(\A_{n})$ where
$b$ is a pair of crossing diagonals and $a$ is a vertex or edge. There
are three types of contributions as illustrated in Figure
\ref{types}. It follows from Theorem \ref{t1a} that if $a$ is a vertex
then $T^1_{a -b}(\A_{n}) \ne 0$ iff $a = \{\delta_{ij}\}$ with $j-i
\equiv 3 \bmod n$ and $b = \{\delta_{i+1,j}, \delta_{i,j-1}\}$. These
elements are module generators for the sets \eqref{s1} and \eqref{s2}
above. The first has $\frac{1}{2}n(n-3)(n-4)$ elements and the second
$n(m+1)$ elements. If $a$ is an edge (not containing a vertex as
above) then $T^1_{a -b}(\A_{n}) \ne 0$ iff the two diagonals in $a$
are edges of $Q_b$ and the other two edges of $Q_b$ are edges of the
$n$-gon. There are two possibilities leading to the sets \eqref{s3} of
cardinality $n(n-5)$ and \eqref{s4} of cardinality
$\frac{1}{2}n(n-5)$. This adds up to the formula in the statement.
\end{proof}

\begin{remark} If $n=4$ then the Stanley-Reisner ring $A$ of
$\mathcal{A}_4 \ast \Delta_m$ is $k[x_0, x_1, y_0, \dots ,
y_m]/(x_0x_1)$, thus the base space of $\Proj A$ is smooth of
dimension $\frac{1}{2}(m+1)(m+2)$.
\end{remark}

Even though the base space is smooth it is a non-trivial task to
compute the versal family. We give it only in the case $n=6$ as an
example. The computations are done by lifting equations and relations
using the program Maple. The $D_6$ symmetry helps shorten the task.
In this case the sets \eqref{s1} -- \eqref{s4} have cardinality $18$,
$6(m+1)$, $6$ and $3$. Set
$$h_{i,i+1} = x_{i-1,i+2}(t_{i,1}x_{i-1,i+2} +  t_{i,2}x_{i-1,i+3} +
t_{i,3}x_{i-2,i+2} + \sum_{k=0}^m r_{i,k}y_k)$$ for $i = 1, \dots ,
6$, where the $t_{i,l}$ and $r_{i,k}$ are parameters. Let $u_i$ and
$r_i$ be parameters dual to the sets \eqref{s3} and \eqref{s4}. To
avoid the non-algebraic deformations in the case $m=-1$ we use the
functor $\Def^{a}_X$, see \cite[Section 6]{ac:def}.

\begin{proposition} \label{a6} The versal algebraic family of the
Stanley-Reisner scheme of $\mathcal{A}_6 \ast \Delta_m$ is defined by
the 15 equations
\begin{multline*} x_{{i,i+2}}x_{{i+1,i-1}}+h_{{i,i+1}}x_{{i+2,i-1}} +
s_{{i+1}}h_{{i+1,i+2}}h_{{i,i-1}}- u_{{i}}h_{{i+1,i+2}}x_{{i+3,i-1}}
\\ \shoveright{- u_{{i+1}}h_{{i,i-1}}x_{{i+2,i-2}}
-u_{{i}}u_{{i+1}}h_{{i-2 ,i-1}}h_{{i+2,i+3}}, \quad i=1, \dots ,6} \\
\shoveleft{x_{{i,i+3}}x_{{i+1,i-1}}-s_{{i}}h_{{i,i+1}}x_{{i+3,i-1}}-s_{{i+1}}x_{{i+1,i+3}}h_{{i,i-1}}+u_{{i}}x_{{i+1,i+3}}x_{{i+3,i-1}}}\\
+u_{{i+3}}(
u_{{i-1}}h_{{i,i+1}}x_{{i,i+2}}+u_{{i+1}}h_{{i,i-1}}x_{{i,i-2}})
-u_{{i+1}}s_{{i}}h_{{i,i-1}}h_{{i+3,i-2}} - u_
{{i+3}}s_{{i+2}}h_{{i,i+1}}h_{{i,i-1}}\\
\shoveright{-u_{{i-1}}s_{{i+1}}h_{{i,i+1}}h_{{i+2,i+3}}-u_{{i}}u_{{i+1}}u_{{i-1}}h_{{i+2,i+3}}h_{{i+3,i-2}},
\quad i=1, \dots ,6} \\
x_{{i,i+3}}x_{{i+2,i-1}}+s_{{i}}x_{{i,i+2}}x_{{i+3,i-1}}
-s_{{i+1}}s_{{i+2}}h_{{i,i-1}}h_{{i+2,i+3}}-u_{{i
}}u_{{i+2}}x_{{i+3,i-1}}^{2}-u_{{i+3}}u_{{i-1}}x_{{i,i+2}}^{2}\\
+x_{{i+3,i-1}}(
u_{{i}}s_{{i+2}}h_{{i+2,i+3}}+u_{{i+2}}s_{{i+1}}h_{{i,i-1}})
+x_{{i,i+2}}
(u_{{i+3}}s_{{i+2}}h_{{i,i-1}}+u_{{i-1}}s_{{i+1}}h_{{i+2,i+3}} ) \\
-s_{{i}}u_{{i+1}}u_{{i-2}}h_{{i+2,i+3}}h_{{i,i-1}} -u_{{i+1}}u_{{
i-2}}
(u_{{i+3}}u_{{i+2}}h_{{i,i-1}}^{2}+u_{{i-1}}u_{{i}}h_{{i+2,i+3}}^{2} )
\quad i=1, \dots ,3 \end {multline*} (all indices are modulo $6$
except the $s_i$ which must be taken modulo $3$) over the smooth space
with parameters $t_{1,1}, t_{1,2}, t_{1,3}, \dots , t_{6,1}, t_{6,2},
t_{6,3}, r_{1,0}, \dots , r_{6,m}, u_1, \dots ,u_6, s_1, s_2, s_3$.
\end{proposition}

\begin{remark} The equations are written so that the first 3 terms
define linear sections of $G(2,6)$ with the standard Pl\"{u}cker
relations. Note this is achieved over the subspace where all
$u_i=0$. In general one gets easily constructed deformations to
sections of $G(2,6)$ when one omits the first order deformations in
set \eqref{s3}. On the other hand taking them along complicates
matters extremely.
\end{remark}

\section{Unobstructed spheres via starring and unstarring} Stellar
subdivisions are related to deformations of Stanley-Reisner rings via
\begin{proposition}\label{stellar} If $\K$ is a simplicial complex and
$\K^\prime$ is a stellar subdivision of $\K$, then $A_{\K^\prime}$
deforms to $A_{\K}$.  In particular, if $T^2_{A_{\K^\prime}} = 0$ then
$T^2_{A_{\K}} = 0$.
\end{proposition}
\begin{proof} We may assume that $\K^\prime$ is obtained from starring
a simplex $a$ of $X$ at a new vertex $v$, i.e.\ $\K^\prime =
(\K\setminus \overline{\st}(a,X)) \cup \partial a \ast v \ast
\link(a,\K)$.  Thus $\link(v,\K^\prime) = \partial a \ast L$, where $L
= \link(a,\K)$.  We get $T^1_{v-a}(\K^\prime) =
T^1_{\emptyset-a}(\partial a \ast L) \simeq H^0(\partial a \ast L) =
k$ by Theorem~\ref{topopen}.

The corresponding first order deformation with parameter say $t$ is
unobstructed since any obstructed $t^k$ would have to be in a
multigraded part of $T^2$ which vanishes since $k{\mathbf a}\not\in
\{0,1\}^n$.  Since $\partial a \in \K^\prime$, $x_a$ is a generator of
the ideal and the deformation is achieved by perturbing this monomial
to $x_a - tx_v$.
\end{proof}

We will refer to the opposite procedure of starring in a vertex in a
face, i.e.\ a stellar exchange of the form $\flip_{v,b}(\K)$, as
\emph{unstarring} a vertex. Note that only vertices with special links
may be unstarred. We may unstar a vertex $v \in \K$ to form
an edge in $\K^\prime$ if and only if  $\link(v) = S^0 \ast L$ for some
subcomplex $L$ and the vertices of the $S^0$ must be a non-edge of
$\K$. The Stanley-Reisner ideal of $\K^\prime$ is gotten from the
ideal of $\K$ by removing all monomials containing $x_v$ and the
monomial $x_ux_v$. In particular if $\K$ is a flag complex then so is
$\K^\prime$. We will say the $\K^\prime$ is gotten from an \emph{edge
unstarring}.

Let $r,s$ be positive numbers such that $n > r \ge 4$ and $r+s = n+3$. For
every such pair we will realize $\A_n$ as a stellar subdivision of
$\A_r \ast \A_s$. This is a generalization of the process yielding
$T_{9} = \A_6$ from $T_7 = \A_5 \ast \A_4$ described in Section
\ref{sec:tri}.  In general there will be several different series of
unstarrings from $\A_n$ to $\A_r \ast \A_s$ producing many different
combinatorial spheres with unobstructed Stanley-Reisner rings.

Let $\A_r$ be the complex of non-crossing diagonals of the polygon
with vertices $(1,2,\dots , r)$ in that cyclic order. Let $\A_s$ be the
same for the polygon with vertices $(1,2,r,r+1, \dots, n-1, n)$ in
that cyclic order. With this indexing the vertices $\delta_{ij}$ of
$\A_r$ and $\A_s$ are disjoint, but may be interpreted as vertices of
$\A_n$. The difference in vertex sets is the set $$D_{n,r} =
\{\delta_{ij}: 3 \le i \le r-1, \, r+1 \le j \le n\} $$ of
$(n-r)(r-3)$ diagonals in $\A_n$. On the other hand $\A_r \ast \A_s$
contains $(n-r)(r-3)$ edges not in  $\A_n$, namely the edges in
$E_{n,r} =\{ \{\delta_{1,i}, \delta_{2,j}\} : 3 \le i \le r-1, \, r+1 \le j
\le n\}$.

Consider the partial order on $D_{n,r}$ given by $\delta_{ij} >
\delta_{kl}$ if $j-i > l-k$. Note that $2 \le j-i \le n-3$ for
$\delta_{ij} \in D_{n,r}$ and that there is a unique maximal element $\delta_{3,n}$
and a unique minimal element $\delta_{r-1,r+1}$. Let $>$ be any total
order on $D_{n,r}$ which extends this partial order.

\begin{theorem} \label{unstar} Let $r,s$ be positive numbers such that
  $n > r \ge 4$
and $r+s = n+3$.  Successively unstarring the vertices in the totally
ordered set $(D_{n,r},>)$ by starting with the maximal element and
then following the order, realizes $\A_n$ as a stellar subdivision of
$\A_r \ast \A_s$. The sequence of intermediate
simplicial complexes yields $(n-r)(r-3)$ triangulated $(n-4)$-spheres
whose Stanley-Reisner ring has trivial $T^2$.
\end{theorem}
\begin{proof} By the above remarks it is enough to prove that at each
step the chosen $\delta_{ij}$ is unstarrable and that unstarring it
adds one of the missing edges in $E_{n,r}$. We do this by
induction. Clearly $\link(\delta_{3,n}, \A_n) = \{\{\delta_{1,3}\},
\{\delta_{2,n}\}\} \ast \A_{n-1}$. We proceed to prove that at each
step we may do the unstarring $\flip_{a,b}$ with $a = \{\delta_{ij}\}$
and $b= \{\delta_{1,i}, \delta_{2,j}\}$.

Let $\K$ be the result of unstarring up to $\delta_{ij}$. We must
prove that $\link(\delta_{ij},\K)$ is of the form $\{\{\delta_{1,i}\},
\{\delta_{2,j}\}\} \ast L$. This is the same as saying that the only
non-edge in $\link(\delta_{ij},\K)$ containing either $\delta_{1,i}$
or $\delta_{2,j}$ is $\{\delta_{1,i},\delta_{2,j}\}$. The original
link of $\delta_{ij}$ in $\A_n$ was $\A_{i+n-j+1} \ast \A_{j-i+1}$
corresponding to the diagonal splitting the $n$-gon in two. The latter
of these has not changed during the previous unstarrings and we may
disregard it. 

In the former the diagonals crossing $\delta_{1,i}$ were
$\delta_{k,l}$ for $2 \le k \le i$ and $j \le l \le n$. During the
unstarring these have been removed except for those with $k=2$, but
these have been put into edges with $\delta_{1,i}$ except for the case
$l=j$. Thus $\{\delta_{1,i},\delta_{2,j}\}$ is the only non-edge
containing $\delta_{1,i}$. The same argument works for $\delta_{2,j}$
and we have proven the result.
\end{proof}

\begin{remark} In fact one can say more, in the notation of the proof,
one can prove that
$$\link(\delta_{ij},\K) \simeq \{\{\delta_{1,i}\},\{\delta_{2,j}\}\}
\ast \A_i \ast \A_{n-j+3} \ast \A_{j-i+1}\, .$$ This allows us to
compute the $f$-vector at each step. We give here only a formula for
the number of facets. To ease the indexing we set $a_n= c_{n-2}$ (the
Catalan number) to be
the number of facets of $\A_n$.

Let $D_{n,r}^{>k}$ be the set of diagonals that have been unstarred
before coming to step number $k$ in the process. One unstarring
decreases the number of facets by half the number of facets in
$\link(\delta_{ij})$. The number of facets of the complex at step $k$
is therefore
$$a_n - \sum_{\delta_{ij} \in D_{n,r}^{>k}} a_ia_{n-j+3} a _{j-i+1}\,
.$$ Applying the formula to different total orders on $D_{n,r}$ shows
that different orders will in general yield different intermediate
triangulated spheres.
\end{remark}

\begin{example} \label{needed} Clearly one could always chose the
order $(i,j) \le (k,l)$ if $j<l$ or $j=l$ and $i \ge k$.
\end{example}

We may iterate the splitting in Theorem \ref{unstar} to get more
unobstructed Stanley-Reisner rings. Note that the final object $\A_{4}
\ast \A_{4} \ast \dots \ast \A_{4}$ is the boundary of the hyper-octahedron, the join
of $S^0$ with itself $n-3$ times.
\begin{corollary}\label{hypero} If $r_1, \dots ,r_m$ are integers with
$n > r_i \ge 4$ and
$$\sum_{i=1}^m r_i = n + 3(m-1)$$
then $\A_n$ is as a stellar subdivision of $\A_{r_1} \ast \A_{r_2}
\ast \dots \ast A_{r_m}$. In particular $\A_n$ and all joins $\A_{r_1}
\ast \A_{r_2} \ast \dots \ast A_{r_m}$ are stellar subdivisions of the
boundary complex of the 
$n-4$ dimensional hyper-octahedron. The subdivisions are done by
edge-starrings and yield intermediate $(n-4)$-spheres whose
Stanley-Reisner ring has trivial $T^2$.
\end{corollary}

\begin{example} \label{dim2} In dimension 2, i.e.\ $n=6$ the possible
 splittings are only $(4,5)$ corresponding to $T_7$ (the process in
 Theorem \ref{unstar} also yields $T_8$) and $(4,4,4)$ corresponding
 to the octahedron $T_6$. For $n=7$ the possibilities are $(4,6),
 (5,5), (4,4,5), (4,4,4,4)$.
\end{example}

In dimension 2 we saw that the unobstructed $T_{10}$ was gotten from
$T_9 = \A_6$ by an edge starring. We would like to generalize also
this construction to higher dimensions to derive unobstructed
triangulated spheres by starring vertices into an edge of $\A_n$. If
we start with a flag complex $\K$ and star a vertex $v$ in an edge
$\{u,w\}$ to get $\K^\prime$, then the Stanley-Reisner ideal of
$\K^\prime$ is gotten from the ideal of $\K$ by adding $x_ux_w$ and
all monomials $x_{v^\prime}x_v$ where $v^\prime \notin \link(\{u,w\},
\K)$. Therefore $\K^\prime$ is also a flag complex.

In dimension 3 or higher experimentation shows that there are
\emph{very many} series of edge starrings starting in $\A_n$ leading
to unobstructed Stanley-Reisner rings and we have not been able to
find a suitable presentation of them.

We have found one general series, but there are many others. Consider
the series of starrings where we successively star vertices
$\varepsilon_k$, $4 \le k \le n-2$, into the edges $\{\delta_{1,3},
\delta_{k,n}\}$ of $\A_n$ and let $\mathcal{C}_n$ be the end result.

\begin{theorem} The module $T^2_{A_{\mathcal{C}_n}} = 0$ for all $n$.
Successively starring vertices into the edges $\{\delta_{1,3},
\delta_{k,n}\}$ of $\A_n$ yields a sequence of $n-5$ triangulated
$(n-4)$-spheres whose Stanley-Reisner ring has trivial $T^2$.
\end{theorem}

We omit the long and technical proof which consists of a careful case
by case check of what links of vertices and the $L_b$ for
$\mathcal{C}_n$ look like. Instead we illustrate what can
happen in higher dimensions by presenting a complete list of 74
combinatorial 3-spheres which appear as successive edges starrings of
$\A_7$ and have   Stanley-Reisner ring with trivial $T^2$.

The list was constructed as follows. A necessary condition for a
3-sphere $\K$ to be a flag complex and have $T^2(\K) = 0$ is that the
links of all edges must be 4-gons or 5-gons, otherwise $T^2_{a-b}(\K)
\ne 0$ for an edge $a$ with valency $\ge 6$ and suitable $b$. Thus,
from a $\K$ with $T^2 = 0$, to get $\K^\prime$ with $T^2(\K^\prime) =
0$ by edge starring the edge $a$ we star in must satisfy
\begin{equation}\label{cond}\text{for all edges $e \in \link(a,\K)$,
$\link(e,\K)$ is a 4-gon.}
\end{equation} This is because for such an $e$, $\link(e,\K^\prime) =
\sta(a,\link(e,\K))$. Call edges satisfying \eqref{cond} 
\emph{legal} edges.

Most of the computations are done in Maple. Assume after successive
edge starrings we have found a flag complex 3-sphere $\K$ with
$T^2=0$. We may compute the automorphism group as the automorphisms of
the edge graph, and this we do with {\tt polymake}
(\cite{gj:po}). Finding the legal edges may be done in Maple and we
choose one for each orbit of the automorphism group. We compute again
in Maple the result of starring in one of these edges and check with {\tt
polymake} if this is isomorphic to a complex we all ready have.

Doing this in a systematic manner we get the list of 74 3-spheres in 
Table \ref{list1} below. The ``Comes from'' column explains which
edges in which of the previous complexes one may star to get this
sphere. Vertices are the original $\delta_{ij}$ of  $\A_7$ and new vertices
$v_1,\dots ,v_8$ where the index denotes at which step they appear in
the starring process.

To ensure that $T^2=0$ we still have to check that
$T^2_{\emptyset - b} = 0$ for all non-edges $b$. It is enough to do
this for the ``final'' ones, i.e.\ those with no legal edges by
Proposition \ref{stellar}. These
are written with boldface in the tables. We use the identity
$T^2_{\emptyset - b} \simeq H_1(L_b)$ from \cite[Proposition
4.8]{ac:def} to do this in Maple. In all cases $T^2_{\emptyset - b}$
did vanish. Based on this and the dimension 2 case we make the
following conjecture.

\begin{conjecture} If $\K$ is a combinatorial sphere and a flag
complex with $\link(f,\K)$ a $4$-gon or $5$-gon for all codimension 2
faces $f$, then $T^2_{A_{\K}} = 0$.
\end{conjecture}

In the table, we also include $-\chi(\Theta)$, the number of virtual
moduli which a potential smoothing of the Fano fourfold
$\PP(\K*\Delta_0)$ would have, where $\K$ is a sphere in the
table. This is computed by using \cite[Theorem 
12]{ac:cot} coupled with \cite[Proposition 2.1]{ci:hil}. The
Hilbert polynomial of $\PP(\K*\Delta_0)$ may be computed from the
table. It is a function of the
$f$-vector of $\K*\Delta_0$ which again by the Dehn-Sommerville
equations may be computed from the number of vertices and facets in
$\K$. One computes that the Hilbert polynomial of $\PP(\K*\Delta_0)$
is $$\frac{1}{24}f_3t^4 +\frac{1}{12}f_3t^3
+\left(\frac{1}{2}f_0-\frac{1}{24}f_3\right)t^2+\left(\frac{1}{2}f_0-\frac{1}{12}f_3\right)t+ 1
$$
where $f_0$ is the number of vertices and $f_3$ is the number of
facets in the sphere $\K$.   

\begin{remark} In light of Corollary \ref{hypero} the above process
should be implemented starting with the first unobstructed flag
complex, namely the boundary complex of the hyper-octahedron. This
will certainly lead to many more 3-dimensional combinatorial spheres
having Stanley-Reisner ring with trivial $T^2$.
\end{remark}

\newgeometry{lmargin=.3in,tmargin=1in,bmargin=1in,rmargin=.3in}

\begin{table}
\centering \subfloat{\footnotesize \begin{tabular}[t]{|c|l|c|l|c|}    
\hline Vertices &Name & Facets & Comes from &$-\chi(\Theta)$\\
\hline \hline
\multirow{2}{*}{15} & $K_1$ &  47 & $\{\delta_{13}, \delta_{46}\} \in
\A_7$ &34\\
\cline{2-5}
& $K_2$&  46 & $\{\delta_{13}, \delta_{47}\} \in \A_7$ &44\\
\hline 
\multirow{9}{*}{16} & $K_3$ &  51 & $\{\delta_{13}, \delta_{47}\} \in K_1$&38 \\
\cline{2-5}
& \multirow{2}{*}{$K_4$} & \multirow{2}{*}{51} & $\{\delta_{14},
\delta_{57}\} \in K_1$ &\multirow{2}{*}{38}\\
&&& $\{\delta_{24}, \delta_{57}\} \in K_2 $& \\
\cline{2-5}
& $K_5$ &  52 & $\{\delta_{16}, \delta_{24}\} \in K_1$&28 \\
\cline{2-5}
& \multirow{2}{*}{$K_6$}  & \multirow{2}{*}{51} & $\{\delta_{16}, \delta_{25}\} \in
K_1 $&\multirow{2}{*}{36} \\
  & & & $\{\delta_{16}, \delta_{35}\} \in K_2 $& \\
\cline{2-5}
& $K_7$  &  52 & $\{\delta_{13}, \delta_{57}\} \in K_2$&38 \\
\cline{2-5}
& $K_8$  &  51 & $\{\delta_{14}, \delta_{57}\} \in K_2$ &36\\
\cline{2-5}
& $K_9$  &  50 & $\{\delta_{16}, \delta_{25}\} \in K_2$&46 \\
\hline
\multirow{28}{*}{17} &  \multirow{3}{*}{$K_{10}$} &
\multirow{3}{*}{55} & $\{v_1, \delta_{47}\} \in K_3$ & \multirow{3}{*}{40} \\
& & & $\{v_1, \delta_{46}\} \in K_7$ &\\
& & & $\{\delta_{13}, \delta_{46}\} \in K_9$ & \\
\cline{2-5}
& $K_{11}$  &  56 & $\{\delta_{13}, \delta_{57}\} \in K_3$&32 \\
\cline{2-5}
&\multirow{2}{*}{$K_{12}$}  &  \multirow{2}{*}{56}  & $\{\delta_{14},
\delta_{57}\} \in K_3$ & \multirow{2}{*}{32}\\
& & & $\{\delta_{13}, \delta_{47}\} \in K_4$& \\
\cline{2-5}
&\multirow{2}{*}{$K_{13}$}    &   \multirow{2}{*}{56} &
$\{\delta_{16}, \delta_{24}\} \in K_3$& \multirow{2}{*}{32} \\
& & & $\{\delta_{13}, \delta_{47}\} \in K_5$ &\\
\cline{2-5}
& \multirow{2}{*}{$K_{14}$}  &   \multirow{2}{*}{55} & $\{\delta_{16},
\delta_{25}\} \in K_3$ &\multirow{2}{*}{40}\\
& & & $\{\delta_{13}, \delta_{47}\} \in K_6$& \\
\cline{2-5}
&  \multirow{2}{*}{$K_{15}$}  &   \multirow{2}{*}{56} &
$\{\delta_{16}, \delta_{35}\} \in K_3$&\multirow{2}{*}{32} \\
& & & $\{\delta_{16}, \delta_{25}\} \in K_5$ &\\
\cline{2-5}
&\multirow{2}{*}{$K_{16}$}  &   \multirow{2}{*}{55}   &
$\{\delta_{27}, \delta_{36}\} \in K_3$&\multirow{2}{*}{42} \\
& & & $\{\delta_{37}, \delta_{46}\} \in K_4$ &\\
\cline{2-5}
& $K_{17}$  &  55 & $\{\delta_{37}, \delta_{46}\} \in K_3$ &44\\
\cline{2-5}
& \multirow{2}{*}{$K_{18}$}  &  \multirow{2}{*}{56}  & $\{\delta_{16},
\delta_{24}\} \in K_4$ &\multirow{2}{*}{34}\\
& & & $\{\delta_{14}, \delta_{57}\} \in K_5$ &\\
\cline{2-5}
&  \multirow{3}{*}{$K_{19}$}  &  \multirow{3}{*}{56}  &
$\{\delta_{24}, \delta_{57}\} \in K_4$ &\multirow{3}{*}{32}\\
& & & $\{\delta_{14}, \delta_{57}\} \in K_7$ &\\
& & & $\{\delta_{13}, \delta_{46}\} \in K_8$ &\\
\cline{2-5}
&  \multirow{2}{*}{$K_{20}$}  &   \multirow{2}{*}{55}  &
$\{\delta_{27}, \delta_{36}\} \in K_4$ &\multirow{2}{*}{42}\\
& & & $\{\delta_{24}, \delta_{57}\} \in K_9$ &\\
\cline{2-5}
& $K_{21}$  &  57 & $\{\delta_{16}, \delta_{46}\} \in K_5$&24 \\
\cline{2-5}
&  \multirow{2}{*}{$K_{22}$}  &   \multirow{2}{*}{56} &
$\{\delta_{16}, \delta_{24}\} \in K_6$&\multirow{2}{*}{32} \\
& & & $\{\delta_{16}, \delta_{35}\} \in K_7$ &\\
\cline{2-5}
& $K_{23}$  &  56 & $\{\delta_{13}, \delta_{57}\} \in K_8$ &32\\
\cline{2-5}
&  \multirow{2}{*}{$K_{24}$}  &  \multirow{2}{*}{55} & $\{\delta_{26},
\delta_{35}\} \in K_8$ &\multirow{2}{*}{38}\\
& & & $\{\delta_{26}, \delta_{35}\} \in K_9$& \\
\hline
\end{tabular} }
\quad
\subfloat{\footnotesize
\begin{tabular}[t]{|c|l|c|l|c|}    
\hline Vertices &Name & Facets & Comes from &$-\chi(\Theta)$\\
\hline \hline
\multirow{46}{*}{18} &   \multirow{5}{*}{$K_{25}$} &
\multirow{5}{*}{60} & $\{\delta_{14}, \delta_{57}\} \in K_{10}$ &\multirow{5}{*}{36}\\
& & & $\{\delta_{27}, \delta_{36}\} \in K_{12}$ &\\
& & & $\{\delta_{14}, \delta_{57}\} \in K_{16}$ &\\
& & & $\{\delta_{27}, \delta_{36}\} \in K_{19}$ &\\
& & & $\{\delta_{13}, \delta_{47}\} \in K_{20}$ &\\
\cline{2-5}
&   \multirow{2}{*}{$K_{26}$}  &    \multirow{2}{*}{59} &
$\{\delta_{16}, \delta_{25}\} \in K_{10}$&\multirow{2}{*}{42} \\
& & & $\{v_1, \delta_{47}\} \in K_{14}$ &\\
\cline{2-5}
&  \multirow{3}{*}{$K_{27}$}  &   \multirow{3}{*}{60} &
$\{\delta_{16}, \delta_{35}\} \in K_{10}$ &\multirow{3}{*}{34}\\
& & & $\{v_1, \delta_{47}\} \in K_{15}$ &\\
& & & $\{v_2, \delta_{24}\} \in K_{22}$ &\\
\cline{2-5}
&  \multirow{2}{*}{$K_{28}$}  &   \multirow{2}{*}{60} &
$\{\delta_{37}, \delta_{46}\} \in K_{10}$ &\multirow{2}{*}{36}\\
& & & $\{v_1, \delta_{37}\} \in K_{17}$ &\\
\cline{2-5}
& $K_{29}$  &  61 & $\{\delta_{14}, \delta_{57}\} \in K_{11}$ &28\\
\cline{2-5}
&  \multirow{2}{*}{$K_{30}$}  &   \multirow{2}{*}{61} &
$\{\delta_{16}, \delta_{35}\} \in K_{11}$&\multirow{2}{*}{28} \\
& & & $\{\delta_{13}, \delta_{57}\} \in K_{15}$ &\\
\cline{2-5}
&  \multirow{4}{*}{$K_{31}$}  & \multirow{4}{*}{60}  & $\{\delta_{37},
\delta_{46}\} \in K_{11}$&\multirow{4}{*}{38} \\
& & & $\{\delta_{37}, \delta_{46}\} \in K_{12}$& \\
& & & $\{\delta_{37}, \delta_{46}\} \in K_{16}$ &\\
& & & $\{\delta_{13}, \delta_{57}\} \in K_{17}$ &\\
\cline{2-5}
&  \multirow{3}{*}{$K_{32}$}  &   \multirow{3}{*}{61} &
$\{\delta_{13}, \delta_{57}\} \in K_{12}$&\multirow{3}{*}{28} \\
& & & $\{\delta_{14}, \delta_{57}\} \in K_{13}$ &\\
& & & $\{\delta_{13}, \delta_{47}\} \in K_{18}$ &\\
\cline{2-5}
&  \multirow{2}{*}{$K_{33}$}  &   \multirow{2}{*}{61} &
$\{\delta_{24}, \delta_{57}\} \in K_{12}$&\multirow{2}{*}{28} \\
& & & $\{\delta_{13}, \delta_{47}\} \in K_{19}$ &\\
\cline{2-5}
& \multirow{2}{*}{$K_{34}$}  &  \multirow{2}{*}{60} & $\{\delta_{16},
\delta_{25}\} \in K_{13}$ &\multirow{2}{*}{36}\\
& & & $\{\delta_{26}, \delta_{35}\} \in K_{15}$& \\
\cline{2-5}
& \multirow{3}{*}{$K_{35}$}  &  \multirow{3}{*}{60} & $\{\delta_{37},
\delta_{46}\} \in K_{13}$&\multirow{3}{*}{38} \\
& & & $\{\delta_{37}, \delta_{46}\} \in K_{15}$& \\
& & & $\{\delta_{16}, \delta_{24}\} \in K_{17}$ &\\
\cline{2-5}
&  \multirow{3}{*}{$K_{36}$}  &   \multirow{3}{*}{60} &
$\{\delta_{16}, \delta_{24}\} \in K_{14}$&\multirow{3}{*}{36} \\
& & & $\{\delta_{13}, \delta_{47}\} \in K_{22}$ &\\
& & & $\{v_1, \delta_{57}\} \in K_{23}$ &\\
\cline{2-5}
&  \multirow{2}{*}{$K_{37}$}  &   \multirow{2}{*}{60} &
$\{\delta_{16}, \delta_{35}\} \in K_{14}$&\multirow{2}{*}{36} \\
& & & $\{\delta_{37}, \delta_{46}\} \in K_{22}$ &\\
\cline{2-5}
& \multirow{2}{*}{$K_{38}$}  &  \multirow{2}{*}{59} & $\{\delta_{37},
\delta_{46}\} \in K_{14}$&\multirow{2}{*}{46} \\
& & & $\{\delta_{16}, \delta_{25}\} \in K_{17}$& \\
\cline{2-5}
& $K_{39}$  &  60 & $\{\delta_{16}, \delta_{25}\} \in K_{15}$&36 \\
\cline{2-5}
&  \multirow{3}{*}{$K_{40}$}  &   \multirow{3}{*}{60} &
$\{\delta_{27}, \delta_{36}\} \in K_{15}$ &\multirow{3}{*}{38}\\
& & & $\{\delta_{16}, \delta_{35}\} \in K_{16}$ &\\
& & & $\{\delta_{16}, \delta_{25}\} \in K_{18}$ &\\
\cline{2-5}
&  \multirow{2}{*}{$\boldsymbol{K_{41}}$}  &   \multirow{2}{*}{61} &
$\{\delta_{16}, \delta_{46}\} \in K_{18}$&\multirow{2}{*}{30} \\
& & & $\{\delta_{14}, \delta_{57}\} \in K_{21}$ &\\
\cline{2-5}
& $K_{42}$  &  61 & $\{\delta_{15}, \delta_{24}\} \in K_{22}$&26 \\
\cline{2-5}
& {$\boldsymbol{K_{43}}$}   &  61 & $\{\delta_{16}, \delta_{46}\} \in
K_{22}$ &28\\
\cline{2-5}
&  \multirow{2}{*}{$K_{44}$}  &   \multirow{2}{*}{60} &
$\{\delta_{26}, \delta_{35}\} \in K_{23}$&\multirow{2}{*}{34} \\
& & & $\{\delta_{13}, \delta_{57}\} \in K_{24}$ &\\
\hline
\end{tabular} }
\caption{3-spheres with $T^2(\K) = 0$ generated by edge starrings.}
\label{list1}
\end{table}

\begin{table}
\ContinuedFloat
\centering
\subfloat{\footnotesize \begin{tabular}[t]{|c|l|c|l|c|}    \hline Vertices &
Name & Facets & Comes from & $-\chi(\Theta)$\\
\hline \hline
\multirow{47}{*}{19} &   \multirow{2}{*}{{$\boldsymbol{K_{45}}$}} &
\multirow{2}{*}{65} &  $\{\delta_{24},\delta_{57}\} \in K_{25}$ &\multirow{2}{*}{32} \\
& & & $\{\delta_{27}, \delta_{36}\} \in K_{33}$ &\\
\cline{2-5}
&   \multirow{3}{*}{$\boldsymbol{K_{46}}$}  &    \multirow{3}{*}{65} &
$\{\delta_{37}, \delta_{46}\} \in  K_{25}$& \multirow{3}{*}{32}\\
& & & $\{\delta_{14}, \delta_{57}\} \in K_{28}$ &\\
& & & $\{\delta_{27}, \delta_{46}\} \in K_{31}$ &\\
\cline{2-5}
&  \multirow{4}{*}{$K_{47}$}  &   \multirow{4}{*}{64} &
$\{\delta_{16}, \delta_{35}\} \in K_{26}$&\multirow{4}{*}{38} \\
& & & $\{\delta_{16}, \delta_{25}\} \in K_{28}$ &\\
& & & $\{v_1, \delta_{47}\} \in K_{37}$ &\\
& & & $\{v_1, \delta_{37}\} \in K_{38}$ &\\
\cline{2-5}
&  \multirow{3}{*}{$K_{48}$}  &   \multirow{3}{*}{64} &
$\{\delta_{16}, \delta_{25}\} \in K_{27}$ &\multirow{3}{*}{38}\\
& & & $\{v_3, \delta_{35}\} \in K_{37}$ &\\
& & & $\{v_1, \delta_{47}\} \in K_{39}$ &\\
\cline{2-5}
& \multirow{4}{*}{$K_{49}$}  &  \multirow{4}{*}{64} & $\{\delta_{26},
\delta_{35}\} \in K_{27}$ &\multirow{4}{*}{38}\\
& & & $\{v_3, \delta_{25}\} \in K_{34}$ &\\
& & & $\{v_3, \delta_{24}\} \in K_{36}$ &\\
& & & $\{v_1, \delta_{57}\} \in K_{44}$ &\\
\cline{2-5}
& \multirow{3}{*}{$K_{50}$}  &  \multirow{3}{*}{65} & $\{\delta_{37},
\delta_{46}\} \in K_{27}$ &\multirow{3}{*}{30}\\
& & & $\{\delta_{16}, \delta_{35}\} \in K_{28}$ &\\
& & & $\{v_1, \delta_{37}\} \in K_{35}$ &\\
\cline{2-5}
&  \multirow{4}{*}{$K_{51}$}  &   \multirow{4}{*}{65} & $\{v_3,
\delta_{14}\} \in K_{29}$ &\multirow{4}{*}{32}\\
& & & $\{v_3, \delta_{13}\} \in K_{32}$ &\\
& & & $\{\delta_{14}, \delta_{57}\} \in K_{34}$& \\
& & & $\{\delta_{26}, \delta_{35}\} \in K_{40}$& \\
\cline{2-5}
&  \multirow{2}{*}{$K_{52}$}  &   \multirow{2}{*}{66} & $\{
\delta_{16}, \delta_{35}\} \in K_{29}$ &\multirow{2}{*}{24}\\
& & & $\{\delta_{14}, \delta_{57}\} \in K_{30}$& \\
\cline{2-5}
 &    {$\boldsymbol{K_{53}}$}   &  66 & $\{\delta_{24},
\delta_{57}\} \in K_{29}$&24 \\
\cline{2-5}
&  \multirow{6}{*}{$K_{54}$}  &   \multirow{6}{*}{65} & $\{
\delta_{37}, \delta_{46}\} \in K_{29}$ &\multirow{6}{*}{34}\\
& & & $\{\delta_{37}, \delta_{46}\} \in K_{30}$ &\\
& & & $\{\delta_{14}, \delta_{57}\} \in K_{31}$ &\\
& & & $\{\delta_{37}, \delta_{46}\} \in K_{32}$ &\\
& & & $\{\delta_{14}, \delta_{57}\} \in K_{35}$ &\\
& & & $\{\delta_{37}, \delta_{46}\} \in K_{40}$ &\\
\cline{2-5}
&  \multirow{2}{*}{$K_{55}$}  &   \multirow{2}{*}{65} & $\{
\delta_{16}, \delta_{25}\} \in K_{30}$&\multirow{2}{*}{32} \\
& & & $\{\delta_{13}, \delta_{57}\} \in K_{39}$ &\\
\cline{2-5}
& \multirow{2}{*}{$K_{56}$}  &  \multirow{2}{*}{65} & $\{ \delta_{26},
\delta_{35}\} \in K_{30}$ &\multirow{2}{*}{32}\\
& & & $\{\delta_{16}, \delta_{35}\} \in K_{34}$ &\\
\cline{2-5}
& $K_{57}$  &  66 & $\{ \delta_{16}, \delta_{24}\} \in K_{32}$ &24\\
\cline{2-5}
& \multirow{3}{*}{$K_{58}$}  &  \multirow{3}{*}{64} & $\{ \delta_{37},
\delta_{46}\} \in K_{34}$ &\multirow{3}{*}{42}\\
& & & $\{\delta_{16}, \delta_{25}\} \in K_{35}$& \\
& & & $\{\delta_{26}, \delta_{35}\} \in K_{39}$& \\
\cline{2-5}
&  \multirow{3}{*}{$K_{59}$}  &   \multirow{3}{*}{64} & $\{
\delta_{37}, \delta_{46}\} \in K_{36}$ &\multirow{3}{*}{42}\\
& & & $\{\delta_{37}, \delta_{46}\} \in K_{37}$ &\\
& & & $\{\delta_{16}, \delta_{24}\} \in K_{38}$ &\\
\cline{2-5}
&  \multirow{2}{*}{$K_{60}$}  &   \multirow{2}{*}{65} & $\{
\delta_{26}, \delta_{35}\} \in K_{37}$ &\multirow{2}{*}{30}\\
& & & $\{\delta_{37}, \delta_{46}\} \in K_{42}$ &\\
\cline{2-5}
& \multirow{2}{*}{$K_{61}$}  &  \multirow{2}{*}{64} & $\{ \delta_{27},
\delta_{36}\} \in K_{39}$ &\multirow{2}{*}{42}\\
& & & $\{\delta_{16}, \delta_{25}\} \in K_{40}$& \\
\hline
\end{tabular} }
\quad
\subfloat{\footnotesize \begin{tabular}[t]{|c|l|c|l|c|}    \hline Vertices &
Name & Facets & Comes from & $-\chi(\Theta)$\\
\hline \hline
\multirow{28}{*}{20} &  \multirow{2}{*}{$K_{62}$} &
\multirow{2}{*}{68}  &  $\{v_4,\delta_{35}\} \in K_{47}$&\multirow{2}{*}{40} \\
& & & $\{v_4, \delta_{25}\} \in K_{48}$ &\\
\cline{2-5}
&    \multirow{3}{*}{$K_{63}$} &  \multirow{3}{*}{69} &
$\{\delta_{26},\delta_{35}\} \in K_{47}$&\multirow{3}{*}{32} \\
& & & $\{v_3, \delta_{37}\} \in K_{50}$ &\\
& & & $\{v_1, \delta_{47}\} \in K_{60}$ &\\
\cline{2-5}
&    \multirow{2}{*}{$K_{64}$} &  \multirow{2}{*}{69} &
$\{\delta_{37},\delta_{46}\} \in K_{47}$&\multirow{2}{*}{34} \\
& & & $\{v_1, \delta_{37}\} \in K_{59}$ &\\
\cline{2-5}
&   \multirow{4}{*}{$K_{65}$} & \multirow{4}{*}{68} &
$\{v_4,\delta_{25}\} \in K_{48}$ &\multirow{4}{*}{44}\\
& & & $\{\delta_{16}, \delta_{25}\} \in K_{49}$& \\
& & & $\{v_3, \delta_{25}\} \in K_{58}$ &\\
& & & $\{v_3, \delta_{24}\} \in K_{59}$ &\\
\cline{2-5}
&    \multirow{3}{*}{$K_{66}$} &  \multirow{3}{*}{69} &
$\{\delta_{37},\delta_{46}\} \in K_{48}$ &\multirow{3}{*}{34}\\
& & & $\{\delta_{16}, \delta_{25}\} \in K_{50}$& \\
& & & $\{v_1, \delta_{37}\} \in K_{58}$ &\\
\cline{2-5}
&    \multirow{5}{*}{$\boldsymbol{K_{67}}$} &  \multirow{5}{*}{69}  &  
$\{\delta_{37},\delta_{46}\} \in K_{51}$ &\multirow{5}{*}{38}\\
& & & $\{v_3, \delta_{14}\} \in K_{54}$ &\\
& & & $\{\delta_{37}, \delta_{46}\} \in K_{55}$& \\
& & & $\{\delta_{14}, \delta_{57}\} \in K_{58}$& \\
& & & $\{\delta_{26}, \delta_{35}\} \in K_{61}$ &\\
\cline{2-5}
&    \multirow{2}{*}{$\boldsymbol{K_{68}}$} &  \multirow{2}{*}{70} &  
$\{\delta_{26},\delta_{35}\} \in K_{52}$&\multirow{2}{*}{28} \\
& & & $\{\delta_{14}, \delta_{57}\} \in K_{56}$ &\\
\cline{2-5}
&    \multirow{3}{*}{$\boldsymbol{K_{69}}$} &  \multirow{3}{*}{70} &  
$\{\delta_{37},\delta_{46}\} \in K_{52}$ &\multirow{3}{*}{30}\\
& & & $\{\delta_{16}, \delta_{35}\} \in K_{54}$ &\\
& & & $\{\delta_{37}, \delta_{46}\} \in K_{57}$ &\\
\cline{2-5}
&    \multirow{3}{*}{$\boldsymbol{K_{70}}$} &  \multirow{3}{*}{69} &  
$\{\delta_{26},\delta_{35}\} \in K_{55}$ &\multirow{3}{*}{38}\\
& & & $\{\delta_{16}, \delta_{25}\} \in K_{56}$ &\\
& & & $\{\delta_{16}, \delta_{35}\} \in K_{58}$ &\\
\cline{2-5}
&  $K_{71}$ & 69 & $\{v_3,\delta_{35}\} \in K_{60}$&34 \\
\hline
\multirow{7}{*}{21} &  \multirow{4}{*}{$K_{72}$} & \multirow{4}{*}{73}
& $\{\delta_{26},\delta_{35}\} \in K_{62}$&\multirow{4}{*}{36} \\
& & & $\{v_3, \delta_{37}\} \in K_{64}$ &\\
& & & $\{v_4, \delta_{25}\} \in K_{65}$ &\\
& & & $\{v_4, \delta_{25}\} \in K_{66}$ &\\
\cline{2-5}
&   \multirow{3}{*}{$K_{73}$} &  \multirow{3}{*}{73} 
& $\{v_4,v_6\} \in K_{63}$& \multirow{3}{*}{36}\\
& & & $\{v_3, \delta_{37}\} \in K_{66}$ &\\
& & & $\{v_1, \delta_{47}\} \in K_{71}$ &\\
\hline
\multirow{2}{*}{22} &  \multirow{2}{*}{$\boldsymbol{K_{74}}$} 
& \multirow{2}{*}{77} & $\{v_6,\delta_{26}\} \in K_{72}$&\multirow{2}{*}{38} \\
& & & $\{v_4, v_6\} \in K_{73}$ &\\
\hline
\end{tabular} }
\caption{3-spheres with $T^2(\K) = 0$ generated by edge starrings.}
\end{table}

\restoregeometry 
\bibliographystyle{amsalpha} \bibliography{jan}
\address{Jan Arthur Christophersen\\
Matematisk institutt\\
 Postboks 1053 Blindern\\ 
 University of Oslo\\
N-0316 Oslo, Norway}{christop@math.uio.no}

\address{
Nathan Owen Ilten\\
Department of Mathematics\\
University of California\\
Berkeley CA 94720, USA\\}{nilten@math.berkeley.edu}

\end{document}